\documentclass[a4paper]{amsart}

\usepackage{amscd,amsthm,amsfonts,amssymb,amsmath}
\usepackage[colorlinks=true]{hyperref}
\usepackage{fullpage}
\usepackage[all]{xy}

    \newcommand{\BA}{{\mathbb {A}}} 
    \newcommand{\BC}{{\mathbb {C}}} 
     \newcommand{\BF}{{\mathbb {F}}}

    \newcommand{\BQ}{{\mathbb {Q}}} \newcommand{\BR}{{\mathbb {R}}}

     \newcommand{\BZ}{{\mathbb {Z}}}

    \newcommand{\CA}{{\mathcal {A}}}

     \newcommand{\CH}{{\mathcal {H}}}

    \newcommand{\CO}{{\mathcal {O}}}

     \newcommand{\RN}{{\mathrm {N}}}

    \newcommand{\Gal}{{\mathrm{Gal}}} \newcommand{\GL}{{\mathrm{GL}}}

    \newcommand{\ord}{{\mathrm{ord}}} \newcommand{\rank}{{\mathrm{rank}}}

    \renewcommand{\mod}{\ \mathrm{mod}\ }

    \newcommand{\sign}{{\mathrm{sign}}}

    \newcommand{\tor}{{\mathrm{tor}}}
    \newcommand{\RTr}{{\mathrm{Tr}}}
    
    \newcommand{\Vol}{{\mathrm{Vol}}}

    \font\cyr=wncyr10

    \newcommand{\Sha}{\hbox{\cyr X}}
    \newcommand{\wh}{\widehat}
    
    \newcommand{\pair}[1]{\langle {#1} \rangle}

    \newcommand{\ov}{\overline}

    \newcommand{\lra}{\longrightarrow}
    \newcommand{\ra}{\rightarrow} 
    \newcommand{\bs}{\backslash}
    \newcommand{\nequiv}{\equiv\hspace{-10pt}/\ }
    
\newcommand{\LL}[1]{L^{(alg)}(A^{(#1)}, 1)}

    \theoremstyle{plain}
    \newtheorem{thm}{Theorem}[section] \newtheorem{cor}[thm]{Corollary}
    \newtheorem{lem}[thm]{Lemma}  \newtheorem{prop}[thm]{Proposition}
     \newtheorem{defn}[thm]{Definition}

\theoremstyle{remark} 
\theoremstyle{remark} 
\theoremstyle{remark} 

    \newcommand{\Neron}{N\'{e}ron~}
    
    \newcommand{\idele}{id\'{e}le~}

    \numberwithin{equation}{section}

\begin{document}

\title{Quadratic Twists of elliptic curves}

\author{John Coates, Yongxiong Li, Ye Tian, and Shuai Zhai}

\thanks{John Coates was supported by the Basic Science Research Program through the
National Research Foundation of Korea(NRF), funded by the Ministry of
Education(2013053914).}

\thanks{Ye Tian was supported by NSFC grants 11325106 and 11031004, and the 973 Program 2013CB834202.}


\dedicatory{\ \ \ \ \ \ \ \ \ \ \ \ \ \ \ \ \ \ \ \ \ \ \ \ \ \ \ \ \ \ \ \ \ \ \ \ \ \ \ \ \ \ \ \ \ \ \ \ \ \ \ \ \ \ \ \ \ \ \ \ \ \ \ \ \ \ \ \it For Bryan Birch and Peter Swinnerton-Dyer}

\begin{abstract}
The paper generalizes, for a wide class of elliptic curves defined over $\BQ$, the celebrated classical lemma of Birch and Heegner to quadratic twists by discriminants having any prescribed number of prime factors. In addition, it proves stronger results for the family of quadratic twists of the modular elliptic curve $X_0(49)$, including showing that there is a large class of explicit quadratic twists whose complex $L$-series does not vanish at $s=1$, and for which the full Birch-Swinnerton-Dyer conjecture holds.
\end{abstract}

\maketitle

\tableofcontents

\section{Introduction}
Let $E$ be an elliptic curve defined over $\BQ$, and let $L(E,s)$ be
the complex $L$-series of $E$. For each square free non-zero integer
$d \neq 1$, we write $E^{(d)}$ for the twist of $E$ by the quadratic
extension $\BQ(\sqrt{d})/\BQ$, and $L(E^{(d)},s)$ for its complex
$L$-series. Those $d$ for which $L(E^{(d)},s)$ has a zero at $s=1$ of order at most 1 are particularly interesting because, for such $d$,
we know that the rank of $E^{(d)}(\BQ)$ is equal to the order of this zero at $s=1$, and the Tate-Shafarevich group of $E^{(d)}$ is finite, by the work of Gross-Zagier and Kolyvagin.
It has been conjectured by Goldfeld that, amongst those
$d$ such that $L(E^{(d)},s)$ has root number $+1$, one should have
$L(E^{(d)}, 1) \neq 0$ outside a set of density zero, and
similarly, amongst those $d$ such that $L(E^{(d)},s)$ has root
number $-1$, one should have that $L(E^{(d)}, s)$ has a simple zero
at $s=1$ outside a set of density zero. Little is known about this
phenomenon at present, beyond the classical results of
\cite{Waldspurger}, \cite{BFH}, \cite{RKM} proving that there are
infinitely many $d$ such that $L(E^{(d)}, s)$ does not vanish at
$s=1$, and infinitely many $d$ such that $L(E^{(d)},s)$ has a simple
zero at $s=1$. The aim of the present paper is to make a modest
first step in developing  techniques, which are largely inspired by
the work of one of us \cite{Tian}, \cite{Tian2} for the elliptic
curve $E: y^2 = x^3 - x$, to prove further results in this
direction. Let $C_E$, or simply $C$
when there is no danger of confusion,
denote the conductor of $E$.  As usual, $\Gamma_0(C)$ will denote the
subgroup of $SL_2(\mathbb{Z})$ consisting of all matrices with the
bottom left hand corner entry divisible by $C$, and we write
$X_0(C)$ for the corresponding compactified modular curve. By the theorem of
Wiles for $E$ semistable, and its generalization to all $E$ by Breuil-Conrad-Diamond-Taylor, there is a non-constant rational map
\begin{equation}\label{mp}
f: X_0(C) \to E
\end{equation}
defined over $\BQ$, which we will always assume maps the cusp at infinity, which we denote
by $[\infty]$, to the zero element $O$ of $E$. Write $[0]$ for
the cusp of the zero point in the complex plane, so that $f([0])$ is
a torsion point in $E(\BQ)$ by the theorem of Manin-Drinfeld.
By generalizing an idea going back to Heegner and Birch, we prove the following fairly general result.

\begin{thm}\label{bh} Let $E$ be an elliptic curve over $\BQ$ of conductor
$C = C_E$,  and let $f: X_0(C)\ra E$ be a modular parametrization as in \eqref{mp}. Assume that
\begin{enumerate}
\item $f([0]) \notin 2E(\BQ)$,
\item there is a good supersingular prime $q_1$ for $E$, with $q_1
\equiv 1 \mod 4$, and $C$ a square modulo $q_1$.
\end{enumerate}
If $k$ is any integer $\geq 1$, there are infinitely many square
free integers $M$, having exactly $k$ prime factors, such that $L(E^{(M)}, s)$
has a zero at $s=1$ of order 1. Similarly, if $k$ is any integer $\geq 2$,
there are infinitely many square free integers $M$ having exactly $k$ prime
factors such that $L(E^{(M)},s)$ does not vanish at $s=1$.
\end{thm}
\noindent A necessary condition for the existence of such a good supersingular
prime $q_1$ for $E$ is that the $2$-primary subgroup of $E(\BQ)$ should
have order at most $2$. Here are some examples of curves to which this
theorem applies. Take $E = X_0(14)$ with $q_1 =5$, and $E = X_0(49)$
with $q_1$ any prime which is $\equiv 1 \mod 4$, and which is not a
square modulo $7$. Examples where  the modular parametrization map $f$
is not an isomorphism are given by the two curves with equations
\begin{equation}\label{cr}
y^2 + xy + y = x^3 - x - 1, \, \, \textrm{and} \, \, \, \,  y^2 = x^3 - x^2 - x - 2,
\end{equation}
which have conductors $C=69$ and $C=84$, respectively; one can take
$q_1 = 5$ for the first curve, and $q_1=41$ or $q_1=89$ for the
second curve. Both curves have Mordell-Weil group equal to $\BZ/2\BZ$,
and the value at $s=1$ of the complex $L$-series of each curve is
equal to $\omega(E)/2$, where $\omega(E)$ denotes the least positive real
period of the \Neron differential on the curve. We are very grateful
to John Cremona for pointing out to us why condition (1) of Theorem
\ref{bh} is valid for these last two curves.

\medskip

We quickly recall the conjectural exact Birch-Swinnerton-Dyer formula for any elliptic curve $E$
over $\BQ$ with $L(E, 1) \neq 0$. For such curves $E$, the theorem of Kolyvagin tells us that
both $E(\BQ)$ and the Tate-Shafarevich group $\Sha(E)$ are finite. Let $\omega(E)$ denote
the least positive real period of a \Neron differential on $E$, so that $L(E,1)/\omega(E)$ is a non-zero
rational number. Let $c_\infty(E)$ denote the number of connected components of $E(\BR)$, and
for each prime $q$ dividing $C$, let $c_q(E) = [E(\BQ_q): E_0(\BQ_q)]$, where $\BQ_q$ is the $q$-adic completion
of $\BQ$, and $E_0(\BQ_q)$ is the subgroup of points with non-singular reduction modulo $q$. Then
the full Birch-Swinnerton-Dyer conjecture asserts that, under the assumption that $L(E, 1) \neq 0$, we have
\begin{equation}\label{fbsd}
L(E, 1)/\omega(E) = c_\infty(E) \prod_{q|C} c_q(E) \#(\Sha(E))/\#(E(\BQ))^2.
\end{equation}
We stress that this full Birch-Swinnerton-Dyer conjecture is known at present only for very few elliptic curves $E$.
If $p$ is any prime number, the equality of the powers of $p$ occurring on the two sides of \eqref{fbsd}
will be called the exact $p$-Birch-Swinnerton-Dyer formula. Iwasawa theory does provide a proof of the exact
$p$-Birch-Swinnerton-Dyer formula for all but a finite number of odd primes $p$ once we know that $L(E, 1) \neq 0$.
However, the methods of Iwasawa theory yield nothing at present for the 2-part of the exact formula, and we stress
that it is the 2-part of the Birch-Swinnerton-Dyer formula which is needed for carrying out Tian's induction argument
for quadratic twists.

\medskip

For the remainder of this paper, we let $A$ be the modular curve
$X_0(49)$, which has genus 1, and which we view as an elliptic curve
by taking $[\infty]$ to be the origin of the group law. It is well
known that $A$ has complex multiplication by the ring of integers
$\mathfrak O = \BZ \left[\frac{1+\sqrt{-7}}{2}\right]$ of the field
$F=\BQ(\sqrt{-7})$, and has a minimal Weierstrass equation given by
$$
\qquad y^2 + xy = x^3 - x^2 - 2x -1.
$$
Moreover, $A(\BQ)= \BZ/2\BZ$, and consists of the cusps $[\infty]$
and $[0]=(2, -1)$. The discriminant of $A$ is $-7^3$, the
$j$-invariant of $A$ is $-3^35^3$, and its \Neron differential
$\omega=\frac{dx}{2y+x}$ has fundamental real period
$\Omega_\infty=\frac{\Gamma\left(1/7\right)\Gamma\left
(2/7\right)\Gamma\left(4/7\right)}{2\pi\sqrt{7}}$. Also, a simple
computation shows that $\BQ(A[2])=\BQ(\sqrt{-7})$, and
$\BQ(A[4])=\BQ(i, \sqrt[4]{-7})$. Writing $L(A, s)$ for the complex
$L$-series of $A$, we have
$$
L(A,1)/\Omega_\infty = 1/2.
$$
\noindent Further, it is known that the Tate-Shafarevich group
of $A$ is trivial, and that the conjecture of Birch and
Swinnerton-Dyer is valid for $A$. However, we stress that the 2-part
of the conjecture of Birch and Swinnerton-Dyer is still unknown for
arbitrary quadratic twists of $A$, even when the complex $L$-series
of the twist does not vanish at $s=1$. Note that, for a discriminant
$d$, which is prime to $7$, the curves $A^{(d)}$ and $A^{(-7d)}$ are
isogenous over $\BQ$. It is then easily seen that the root number of
$A^{(d)}$ is $+1$ if and only if either $d > 0$ and $d$ is prime to $7$,
or $d < 0$ and is divisible by 7. We use similar ideas to those of Zhao
\cite{CKLZ} to prove the following two theorems about the values at $s=1$
of the $L$-series of quadratic twists of $A$ with root number $+1$. We
also give a proof of both results by Waldspurger's formula in section 5.
\begin{thm}\label{main3}
Let $R=q_1 \cdots q_r$ be a product of $r \geq 0$ distinct primes,
which are $\equiv 1\ \mod 4$ and inert in the field $F$.
Then $L(A^{(R)}, 1) \neq 0$, $A^{(R)}(\BQ)$ is finite, the
Tate-Shafarevich group of $A^{(R)}$ is finite of odd cardinality,
and the full Birch-Swinnerton-Dyer conjecture is valid for
$A^{(R)}$.
\end{thm}
We remark that the non-vanishing result of this theorem
can be given a completely different proof by the techniques used to
establish Theorem \ref{bh}, but at present we have no idea how to
prove the $2$-part of the conjecture of Birch and Swinnerton-Dyer by
such methods. However, the knowledge of the $2$-part of the conjecture
of Birch and Swinnerton-Dyer for the twists of $A$ in Theorem
\ref{main3} turns out to be vital for the proof of Theorem
\ref{main2} below. For any $r \geq 0$ distinct primes
$q_1, \ldots, q_r$, which are all $\equiv 1 \mod 4$ and
inert in $F$, define the field
\begin{equation}\label{t2}
\mathfrak H = \BQ(A[4], \sqrt{q_1}, \ldots, \sqrt{q_r}) = \BQ(i,
\sqrt[4]{-7}, \sqrt{q_1}, \ldots, \sqrt{q_r}).
\end{equation}
For each square free integer $M$, prime to $7$, with $M\equiv 1 \mod 4$, we define
$$
L^{(alg)}(A^{(M)}, 1) = L(A^{(M)}, 1)/\Omega_\infty(A^{(M)}),
$$
which is well known to be a rational number, where $\Omega_\infty(A^{(M)})$ is the least positive real period of $A^{(M)}$. We will always normalise
the order valuation at $2$ by $ord_2(2) = 1$.
\begin{thm}\label{main4}
Let $R = q_1\cdots q_r$ be a product of $r \geq 0$ distinct primes
$\equiv 1 \mod 4$, which are inert in $F$, and let $N = p_1\cdots p_k$
be a product of $k \geq 1$ distinct primes, all of which split
completely in the field $\mathfrak H$ defined by \eqref{t2}.
Put $M = RN$. Then
\begin{equation}\label{t3}
ord_2(L^{(alg)}(A^{(M)}, 1) )\geq r +2k,
\end{equation}
and, if $L(A^{(M)}, 1) \neq 0$, the $2$-primary subgroup of the
Tate-Shafarevich group of $A^{(M)}$ is non-zero.
\end{thm}

\noindent We will give two proofs of this result, one by Zhao's method, and the other using Waldspurger's formula. In fact, the approach via Waldspurger's formula gives a slightly stronger result (see Theorem \ref{bw}), but this stronger statement is not needed for the proof of the following theorem. For the twists of $A$ with root number $-1$, we use
similar ideas to those developed in \cite{Tian2} and
\cite{Tian}, to show that one can combine Theorems \ref{main3}
and \ref{main4} with the theory of Heegner points, to establish the
following result.

\begin{thm}\label{main2}
Let $l_0$ be a prime number $>3$, which is $\equiv 3 \mod 4$
and is inert in the field $F$. Assume that $q_1, \ldots, q_r$
are distinct rational primes, which are  $\equiv 1 \mod 4$,
and inert in both the fields $F$ and $\BQ(\sqrt{-l_0})$.
Let $k$ be any integer $\geq 0$, and let $p_1, \ldots, p_k$
be distinct primes which all split completely in the field
$\mathfrak H$ defined by \eqref{t2}. Put $N = p_1\cdots p_k$,
$R = l_0q_1\cdots q_r$, and $M = -RN$. Assume that the ideal
class group of the imaginary quadratic field $\BQ(\sqrt{-l_0N})$
has no element of exact order $4$. Then $L(A^{(M)},s)$ has a
simple zero at $s=1$, $A^{(M)}(\BQ)$ has rank one, and the
Tate-Shafarevich group of $A^{(M)}$ is finite of odd cardinality.
\end{thm}

\noindent Note that, in the special case when $k=0$ but $r$ is arbitrary, no hypothesis about the ideal class group
is needed for the statement of the theorem, since $\BQ(\sqrt{-l_0})$ has odd class number.
We also remark that, in the paper \cite{CST}, the assertions of Theorems \ref{main4}
and \ref{main2}, are strengthened to show that both theorems hold under the weaker hypothesis
that the primes $p_1,..., p_k$ split completely in the subfield $\BQ(A[4], \sqrt{R})$ of $\mathfrak H$.
Unfortunately, we still do not know enough at present to
prove that the orders of the Tate-Shafarevich group of the twists of
$A$ in Theorems \ref{main4} and \ref{main2} are as predicted by the
conjecture of Birch and Swinnerton-Dyer.

\medskip

We end this introduction
by saying that, for every elliptic curve $E$ defined over $\BQ$,
we believe there should be some analogues of Theorems \ref{mp},
\ref{main3}, \ref{main4}, and \ref{main2} for the family of quadratic
twists of $E$, and it seems to us to be an important problem to first
formulate precisely what such analogues should be, and then to prove them.

\medskip

In conclusion, we thank Li Cai and John Cremona for some very helpful comments on the questions discussed in this paper.
We also thank the Department of Mathematics and PMI at POSTECH, Korea, for their generous support of this research.

\medskip

\section{Generalization of Birch's Lemma}
The aim of this section is to prove Theorem \ref{bh}. As in the
Introduction, let $E$ be an elliptic curve over $\BQ$ of conductor
$C$, and let $\phi = \sum_{n\geq1}a_nq^n$ be the corresponding
primitive cusp form on $\Gamma_0(C)$. Let $K$ be an imaginary
quadratic field, which, for simplicity, we assume is not equal to $\BQ(i), \BQ(\sqrt{-3})$. We write
$\CO$ for the ring  of integers of $K$. We assume throughout this
section that $K$ satisfies the so called Heegner hypothesis for $E$,
namely that every  prime factor of $C$ splits in $K$. Thus there
exists $\mathfrak C\subset \CO$ such that $\CO/\mathfrak C\cong \BZ/C\BZ$. For
each positive integer $M$ with $(M, C)=1$, let $\CO_M=\BZ+M\CO$ be
the order of $K$ of conductor $M$. Writing $\mathfrak C_M=\mathfrak C\cap \CO_M$,
the point
\begin{equation}\label{hp}
P_M=(\BC/\CO_M\ra \BC/\mathfrak C_M^{-1})
\end{equation}
on $X_0(C)$ is defined over the ring class field $H_M$ of $K$, and is called
a Heegner point of conductor $M$. We recall that $H_M$ is the
abelian extension over $K$ characterized by the property that the
Artin map induces an isomorphism $ \wh{K}^\times/K^\times
\wh{\CO}_M^\times\stackrel{\sim}{\lra}\Gal(H_M/K)$, where
$\wh{K}^\times$ denotes the \idele group of $K$, and
$\wh{\CO}_M^\times$ denotes the tensor product over $\BZ$ of $\CO_M^\times$ with $\wh{\BZ} = \prod_p\BZ_p$.
The Heegner points $P_M$ are related
to the value at $s=1$ of the derivative of the $L$-function by the following generalized
Gross-Zagier formula, which is proven in general by Yuan-Zhang-Zhang in \cite{YZZ}, and
its explicit form used here in \cite{CST}. If $\chi$ denotes an abelian character of $K$,
we write $L(E/K, \chi, s)$ for the complex $L$-series of $E/K$ twisted by $\chi$.

\begin{thm} \label{M5} Let $E$ be an
elliptic curve over $\BQ$ of conductor $C$, and let $f: X_0(C)\ra E$ be a modular
parametrization as in \eqref{mp}. Let $K \neq \BQ(i), \BQ(\sqrt{-3})$ be an imaginary quadratic field of
discriminant $d_K$, and assume that every prime dividing $C$ splits in $K$. Let $\chi$ be any ring class character of $K$ with
conductor $M$, where $M \geq 1$ is such that $(M, Cd_K)= 1.$ Let $P_M$ denote the Heegner point on $X_0(C)$ of
conductor $M$ defined by \eqref{hp}, and put
 $$\displaystyle{P_\chi(f):=\sum_{\sigma \in
\Gal(H_M/K)} f(P_M)^{\sigma} \chi(\sigma),}$$
which lies in the tensor product of $E(H_M)$ with $\BC$. Then
$$L'(E/K,\chi, 1)=\frac{8\pi^2 (\phi, \phi)_{\Gamma_0(C)}}{
\sqrt{|D_K M^2|}}\cdot\frac{ \wh{h}_{K}\left(P_\chi(f)\right)}{ \deg
f},$$
where $\wh{h}_K$ denotes the \Neron-Tate height on $E$ over $K$,
$\phi=\sum_n a_n q^n$ is the primitive eigenform of weight 2 attached to $E$, and the Petersson norm is defined by
$$(\phi, \phi)_{\Gamma_0(C)}=\iint_{\Gamma_0(C)\bs \CH} |\phi(z)|^2
dx dy, \qquad z=x+iy.$$
\end{thm}

For a discussion of the action of various operators on the Heegner points on $X_0(C)$, see \cite{Gross}. In particular, let $\mathfrak P_M$ for the set of all conjugates of $P_M$ under the action of the Galois group of $H_M$  over $K$. Then,  writing $\tau$
for complex conjugation and $w_{C}$ for the Fricke involution,
we have the equality of sets
\begin{equation}\label{h2}
w_C \mathfrak P_M =\tau \mathfrak P_M.
\end{equation}
\noindent  For the remainder of this section we shall always take $\ell_0$ to be any
prime with $\ell_0 > 3$ and $\ell_0 \equiv 3 \mod 4$, and define
\begin{equation} \label{h3}
K = \BQ(\sqrt{-\ell_0}).
\end{equation}
Thus, by classical genus theory, $K$ has odd class number. The
following result is essentially due to Birch.
\begin{thm}\label{bl}
Let $E$ be any elliptic curve defined over $\BQ$ with modular
parametrization \eqref{mp}, and assume that $f([0])$ does not belong
to $2E(\BQ)$.  Let $K$ given by \eqref{h3} be such that every prime
dividing $C$ splits in $K$. Let $\mathfrak C$ be an ideal in $\CO$ such
that $\CO/\mathfrak C\cong \BZ/C\BZ$, and let $P= P_1$ given by \eqref{hp} be the corresponding Heegner point of conductor 1.
Then $y_K=\RTr_{H/K} f(P)$ is of infinite order in $E(K)$.
\end{thm}

\noindent We immediately deduce the following corollary, which implies the assertion of Theorem \ref{bh} when $k=1$.

\begin{cor} Under the same hypotheses as in Theorem \ref{bl}, the complex $L$-function $L(E/K,s)$ of
$E$ over $K$ has a simple zero at $s=1$,
$L(E,s)$ does not vanish at $s=1$,  and $L(E^{(-\ell_0)},s)$ has a simple zero at $s=1$.
\end{cor}
\begin{proof}
Note that  $f\circ w_{C} -\epsilon f$ is a constant morphism, where $\epsilon = \pm1$ is the negative of the sign in the functional equation of the complex $L$-series $L(E,s)$. Thus
for all points $P\in X_0(C)$, we have
$$f(P^{w _C}) - \epsilon f(P)=f([\infty]^{w_{C}})-\epsilon f([\infty])=f([0]).$$
If $\epsilon = 1$, we can take $P$ to be a fixed point of $w_C$, whence it would follow that $f([0]) = O$, contradicting our hypothesis that $f([0]) \notin 2E(\BQ)$. Thus necessarily
$\epsilon = -1$. It then follows from \eqref{h2} that we have
$$\ov{y}_K + y_K=h f([0]),$$
where $h$ denotes the class number of $K$. Then $T = hf([0])$ does not belong $2E(\BQ)$ because $h$ is odd. We now prove  that $y_K$
does not belong to the torsion subgroup of $E(K)$.  Observe that $E(K)[2^\infty] = E(\BQ)[2^\infty]$, because $K/\BQ$ is totally ramified at the prime $l_0$, whereas only primes dividing $2C$ are ramified in the field $\BQ(E[2^\infty])$. If $y_K$ does have finite order, then $t=ay_K\in
E(K)[2^\infty]=E(\BQ)[2^\infty]$, where $a$ denotes any odd positive integer which annihilates all elements of
 odd finite order in $E(K)$. It follows easily that
$$a(\ov{y}_K + y_K)=2t \in 2E(\BQ),$$
which would imply that $T \in 2E(\BQ)$, which is a contradiction. For the proof of the corollary, we note that, since
$y_K$ has infinite order, the theorem of Gross-Zagier tells us that
the complex $L$-series of $E$ over $K$ has a simple zero at $s=1$. Then, as $L(E/K, s) = L(E, s)L(E^{(-\ell_0)},s)$,
with $L(E, s)$ having root number $+1$, and with $L(E^{(-\ell_0)},s)$ having root number $-1$, the second assertion
of the corollary follows.

\end{proof}

We now extend Birch's result to quadratic twists with arbitrarily many prime
factors. It is convenient to introduce the following terminology. We define a
prime $q_1$ to be a {\it sensitive} supersingular prime for the elliptic curve $E$
if (i) $q_1$ is a prime of good supersingular reduction for $E$, (ii) $q_1 \equiv 1 \mod 4$, and
(iii) $C = C_E$ is a square modulo $q_1$.  If $E$ possesses a sensitive supersingular
prime $q_1$, then necessarily $E(\BQ)[2^\infty]$ has order at most $2$, because
reduction modulo $q_1$ is injective on $E(\BQ)[2^\infty]$, and there are $q_1+1$
points with coordinates in $\mathbb{F}_{q_1}$ on the reduced curve. Recall that $\phi=\sum_n a_n q^n$
is the primitive cusp form of weight $2$ for $\Gamma_0(C)$ attached to $E$.

\begin{lem}\label{fs} Assume $E$ possesses a sensitive supersingular prime $q_1$, which is inert in $K$. For each integer
$r \geq 2$, define $\Sigma_r$ to be the set of all prime $q \neq q_1$ such that
(i) $q \equiv 1 \mod 4$, (ii) $a_q \equiv 0 \mod 2^r$, (iii) $(q, C) = 1$ and $C$ is a square modulo $q$,
and (iv) $q$ is inert in $K$. Then $\Sigma_r$ is infinite of positive density in the set of primes.
\end{lem}
\begin{proof}  Put $J = \BQ(\sqrt{C}, E[2^{r}])$, and note that $K \cap J = \BQ$, because $\ell_0$ is totally ramified in $K/\BQ$, and only primes dividing $2C$ can ramify in $J$. Also $q_1$ is unramified in $J$ because $(q_1, 2C)=1$.  Thus, writing $\Delta = Gal(JK/\BQ)$, there will be a unique element $\sigma$ in $\Delta$, whose restriction to $K$ is complex conjugation, and whose restriction to $J$ is the Frobenius automorphism of some prime of $J$ above $q_1$.
Now, assuming $r \geq 2$, we claim that $\Sigma_r$ contains the set $\mathcal S$ of all primes not dividing $2\ell_0q_1C$, whose Frobenius automorphisms in $\Delta$ lie in the conjugacy class of $\sigma$. Granted this assertion, the Chebotarev theorem then shows that $\Sigma_r$ is infinite of positive density in the set of all prime numbers.  We now verify that the primes in this set $\mathcal S$
have all the desired properties. Indeed, as $q_1 \geq 5$ and is supersingular,
we have $a_{q_1}=0$, so that the characteristic polynomial of the Frobenius automorphism of $q_1$ acting on the $2$-adic Tate module $T_2(E)$ is equal to
$X^2 + q_1$. Similarly, the characteristic polynomial of the Frobenius automorphism of a prime $q$ not dividing $2C$ acting on $T_2(E)$ is equal to $X^2 + a_qX + q$. Since
$E[2^r] = T_2(E)/2^rT_2(E)$, we conclude that, for $q$ in our set $\mathcal S$, we must have $a_q \equiv 0 \mod 2^r$ and $q \equiv q_1 \mod 2^r$. Also $q$ is inert in $K$, since $q_1$ is inert in $K$. Finally, $q$ splits in $\BQ(\sqrt{C})$ because $q_1$ splits in this field.
\end{proof}

\noindent

\begin{thm}\label{Birch2} Assume that (i) $f([0]) \notin 2E(\BQ)$, and (ii)
there exists a sensitive supersingular prime $q_1$ for $E$. Let $K$
given by \eqref{h3} be such that every prime dividing $C$ splits in $K$,
and $q_1$ is inert in $K$.
For each integer $r \geq 1$, let $R = q_1q_2 \cdots q_r$, where, for $r \geq 2$,  $q_2, \ldots, q_r$ are any distinct primes in the set $\Sigma_r$
defined in Lemma \ref{fs}.
Then $K(\sqrt{R})$ is a subfield of the ring class field $H_R$.  Writing $\chi_R$ for the character of $K$ attached to this
quadratic extension, define the Heegner point $y_R$ by
$$
y_R=\sum_{\sigma\in \Gal(H_R/K)}\chi_R(\sigma)f(P_R)^\sigma.
$$
Then, for each integer $r \geq 1$, we have $y_R\in 2^{r-1} E(\BQ(\sqrt{-\ell_0R}))^-+E(\BQ(\sqrt{-\ell_0 R}))_\tor$, but $y_R\notin 2^rE(\BQ(\sqrt{-\ell_0R}))^-+E((\BQ(\sqrt{-\ell_0R}))_\tor$.
In particular, $y_R$ is of infinite order.
\end{thm}

\noindent Since $y_R$ has infinite order, it follows from Theorem \ref{M5}, that, under the assumptions of Theorem \ref{Birch2},
$L(E/K, \chi_R, s)$ must have a simple zero at $s=1$. But
$$
L(E/K, \chi_R, s) = L(E^{(-\ell_0R)}, s)L(E^{(R)},s).
$$
Since, by hypothesis, $C$ is a square modulo every prime dividing $\ell _0R$, and, as was shown in the proof of Theorem \ref{bl}, $L(E,s)$ has root number $+1$, it follows easily that
$L(E^{(R)},s)$ and $L(E^{(-\ell_0R)}, s)$ have global root numbers equal to $+1$ and $-1$, respectively. Thus we must have that $L(E^{(R)},1) \neq 0$, and that $L(E^{(-\ell_0R)}, s)$ has a simple zero at $s=1$. Hence the following result follows immediately from the above theorem and the theorem of Kolyvagin-Gross-Zagier.
\begin{cor} Under the same hypotheses as in  Theorem \ref{Birch2}, for all
$R = q_1 \cdots q_r$ with $r \geq 1$ we have (i) the complex $L$-series of $E^{(R)}$ does not vanish at $s=1$, and both $E^{(R)}(\BQ)$ and the Tate-Shafarevich group
of $E^{(R)}$ are finite, and (ii) the complex $L$-series of  $E^{(-\ell_0 R)}$ has a simple zero at $s=1$, $E^{(-\ell_0R)}(\BQ)$
has rank 1, and the Tate-Shafarevich group of $E^{(-\ell_0R)}$ is finite.
\end{cor}
\noindent Note that, since $\Sigma _r$ is infinite when $r \geq 2$, the assertions of Theorem \ref{bh} for $k \geq 2$ follow immediately from this corollary. We note also that numerical examples of curves $E$ and a sensitive supersingular primes $q_1$ to which the above theorem can be applied are given by $E=X_0(14)$, for which $q_1 =5$, $E=X_0(49)$ with any prime $q_1$ such that $q_1 \equiv 1 \mod 4$ and $q_1$ is inert in $F=\BQ(\sqrt{-7})$, and the curves $E$ of conductor $69$ and $84$ given by \eqref{cr},  for which we can take $q_1 =5$, and $q_1=41, 89$, respectively. In each example, we choose the prime number $\ell_0$ so that  $\ell_0$ is $\equiv 3 \mod 4$ and $q_1$ is inert in $K = \BQ(\sqrt{-l_0}).$ Note also that,  for $E = X_0(49)$, and $r \geq 2$, the set $\Sigma _r$ contains all primes which are $\equiv 1 \mod 4$, and inert in both $F=\BQ(\sqrt{-7})$ and $K$.

\medskip

Let $R = q_1\cdots q_r$ be as in the statement of Theorem \ref{Birch2}. Define
\begin{equation}\label{b3}
\mathfrak H_R = K(\sqrt{q_1}, \ldots, \sqrt{q_r}).
\end{equation}
We first establish three preliminary lemmas needed in the proof of Theorem \ref{Birch2}.

\begin{lem}\label{sub} The field $\mathfrak H_R$ is a subfield of $H_R$, and the degree of $H_R$ over  $\mathfrak H_R$ is odd.
Moreover, $E(\mathfrak H_R)[2^\infty] = E(\BQ)[2]$.
\end{lem}
\begin{proof} Let $q$ denote any of the primes $q_1, \ldots, q_r$, and let $h$ denote the class number of $K$.
Since $q$ is inert in $K$, the ring class field $H_q$ of conductor $q$ has degree $(q+1)h$ over $K$, and $ord_2((q+1)h) = 1$
because $h$ is odd, and $q \equiv 1 \mod 4$. Hence $H_q$ contains a unique quadratic extension of $K$, which must be unramified
outside of $q$, and so must be equal to $K(\sqrt{q})$ because $ q \equiv 1 \mod 4$. Then the degree of $H_q$ over $K(\sqrt{q})$
is equal to $(q+1)h/2$, which is odd. Since $H_R$
is the compositum of all of the $H_{q}$ for the primes $q$ dividing $R$, the first assertion of the lemma follows easily. For the second assertion,
we note that $E(\mathfrak H_R)[2^\infty] = E(\BQ)[2^\infty]$ because at least one of the primes $\ell_0, q_1, \ldots, q_r$ must ramify in
every subfield of $\mathfrak H_R$ which is strictly larger than $\BQ$, and only the primes dividing $2C$ may ramify in the field $\BQ(E[2^\infty])$.
Then we use that $E$ has a sensitive supersingular prime to conclude that $E(\BQ)[2^\infty] = E(\BQ)[2]$.
\end{proof}
\begin{lem} Let $\mathfrak P(\mathfrak H_R)$ be the set of conjugates of the point $P_R$ under the action of the Galois group of $H_R$ over $\mathfrak H_R$. Then we have the equality of sets
\begin{equation}\label{h7}
w_{C} \mathfrak P(\mathfrak H_R) = \tau \mathfrak P(\mathfrak H_R),
\end{equation}
where $\tau$ denotes complex conjugation.
\end{lem}
To establish this lemma, we use the well known fact that $w_C(P_R) = (P_R)^{\sigma_\mathfrak C}$, where $\sigma_\mathfrak C$ denotes the Artin symbol of $\mathfrak C$ for the extension $H_R/K$. But, for each prime $q$ dividing $R$, the restriction of $\sigma_\mathfrak C$ to $\BQ(\sqrt{q})$ is the Artin symbol of $C\BZ = N_{K/\BQ}\mathfrak C$ for this extension, and this Artin symbol fixes $\sqrt{q}$
because $q$ is a square modulo $C$.

\medskip

For each positive divisor $D$ of $R$, let $\chi_D$ be the character attached to the extension $K(\sqrt{D})/K$, and define the imprimitive Heegner point $z_D$ in $E(K(\sqrt{D}))$ by
$$
z_D = \sum_{\sigma\in \Gal(H_R/K)}\chi_D(\sigma)f(P_R)^\sigma.
$$
Obviously $z_R = y_R$, but for proper divisors $D$ of $R$, we have the following lemma.
\begin{lem} For all positive divisors $D$ of $R$, define  $b_D = \prod_{q | R/D} a_q$, where the product is taken over all primes $q$ dividing $R/D$. Then we have
\begin{equation}\label{h9}
z_D = b_Dy_D.
\end{equation}
In particular, $b_D = 0$ whenever $q_1$ does not divide $D$, because $a_{q_1} = 0$.
\end{lem}
\begin{proof} Since $K(\sqrt{D})$ is contained in $H_D$, the assertion follows immediately from the following general fact first observed by Kolyvagin.
Let $M$ be any positive integer prime to $C$, and let $p$ be a prime number with $(p, MC) =1$ and $p$
inert in $K$. Then
$$
Tr_{H_{Mp}/H_M}f(P_{Mp}) = a_pf(P_M).
$$
\end{proof}

\medskip

In order to prove our assertions about the Heegner point $y_R$, it is convenient to assume that $T = f([0])$
is of exact order $2$. This can always be achieved by composing $f$ with multiplication by an odd integer on $E$,
and we shall assume for the rest of the proof that we have done this. Define
$$
\psi_R =\sum_{\sigma\in \Gal(H_R/\mathfrak H_R)}f(P_R)^\sigma.
$$
Recall that $f(P^{w_{C}}) + f(P) = T$ for all points $P$ on
$X_0(C)$ since we have assumed that $L(E,s)$ has root number $+1$.
As $[H_R:\mathfrak H_R]$ is odd by the first of the above lemmas,
and $T$ is a point in $E(\BQ)$ of order $2$, we deduce immediately
from the second of the above lemmas that
\begin{equation}\label{h8}
\bar{\psi}_R + \psi_R= T.
\end{equation}
Suppose first that
$r=1$, so that $R = q_1$, and let $\sigma$ denote the non-trivial element of $Gal(K(\sqrt{q_1})/K)$. Then $y_R = \psi_R - \sigma(\psi_R)$. Since $a_{q_1} = 0$, it follows from \eqref{h9} with $D=1$ that $\psi_R + \sigma(\psi_R) = 0$, whence we conclude that $y_R = 2\psi_R$. It now follows from \eqref{h8} and the fact that $2T = 0$ that $\bar{y}_R + y_R = 0$. In view of this last equation and the fact that $\sigma(y_R) + y_R = 0$, we see that $y_R$ lies in $E(\BQ(\sqrt{-\ell_0R}))^{-}$.  Suppose finally that $y_R = 2w + t$
for some $w$ in $E(\BQ(\sqrt{-\ell_0R}))^-$ and a torsion point $t$.
Since $E(\mathfrak H_R)[2^\infty] = E(\BQ)[2]$, it follows that $\psi_R= w + t'$
for some $t'$ in $E(\BQ)[2]$, whence $\bar{\psi}_R + \psi_R = 0$,
contradicting \eqref{h8}. This proves our theorem when $r=1$. Now
suppose that $r > 1$. It is  easy to see that we have the identity
\begin{equation}\label{h6}
y_R+\sum_{D|R, D\neq R} z_D=2^r\psi_R.
\end{equation}
If $D \neq R$, we can write $b_D= 2^re_D$ for some integer $e_D$
because of condition (ii) in Lemma \ref{fs}. Hence, as $z_D = b_Dy_D$ by the lemma above,
$$y_R=2^ru _R, \quad \text{with}\quad u_R= \psi_R-\sum_{D|R, D\neq R} e_D y_D,$$
and we recall that $e_D = 0$ if $q_1$ does not divide $D$. In particular,
it follows that the class of $y_R$ in
$E(K(\sqrt{R}))/2^rE(K(\sqrt{R}))$ maps to zero in
$E(\mathfrak H_R)/2^rE(\mathfrak H_R)$. But we have the inflation-restriction exact
sequence
$$0\lra H^1(Gal(\mathfrak H_R/K(\sqrt{R})), E(\mathfrak H_R)[2^r])\lra
H^1(K(\sqrt{R}), E[2^r])\lra H^1(\mathfrak H_R, E[2^r]),$$
 and the kernel
on the left of this sequence is killed by $2$, because, as remarked
above, we have that $E(\mathfrak H_R)[2^\infty] = E(\BQ)[2]$. It follows that
$2y_R\in 2^rE(K(\sqrt{R}))$. Hence $y_R= 2^{r-1}y+t$ for some
$y\in E(K(\sqrt{R}))$ and $t \in E(\BQ)[2]$. Also, we then have
$y=2u_R+s$ for some $s\in E(\BQ)[2]$. We now show that $y$ belongs
to $E(\BQ(\sqrt{-\ell_0R}))^-$. Let $\sigma$ be an element of
$\Gal(\mathfrak H_R/K)$ which maps $\sqrt{q_1}=-\sqrt{q_1}$, and fixes all
$\sqrt{q_i}$ for $2\leq i\leq r$. We claim that
\begin{equation}\label{h5}
\sigma(y) + y= 0.
\end{equation}
Since $y = 2u_R +s$, and $u_R = \psi_R - \sum_{D|R, D\neq R}e_Dy_D$, it suffices to show that
\begin{equation}\label{h11}
\sigma(\psi_R) + \psi_R = 0, \,  \sigma(y_D) + y_D = 0,
\end{equation}
with the latter equation holding for all positive divisors $D$ of $R$, which are not equal to $R$ and have $e_D \neq 0$.
Now the first equation in \eqref{h11} holds because $\sigma(\psi_R) + \psi_R$ is equal to the trace from $H_R$
to $K(\sqrt{q_2}, \ldots, \sqrt{q_r})$ of $f(P_R)$, and this is zero because $a_{q_1} = 0$.
If $e_D \neq 0$, then $D$ is a positive divisor of $R$ which is divisible by $q_1$. Thus the restriction of $\sigma$ to
$K(\sqrt{D})$ must be the non-trivial element of the Galois group of
this field over $K$, and so $y_D = v_D - \sigma(v_D)$, where $v_D = Tr_{H_D/K(\sqrt{D})}f(P_D)$. But
since $q_1$ divides $D$ and $a_{q_1} = 0$, we must have $v_D + \sigma(v_D) = 0$,
and so $y_D = 2v_D$, whence also $\sigma(y_D) + y_D = 0$, completing the proof of \eqref{h11}
We claim that we also have
\begin{equation}\label{h10}
\bar{y} + y = 0.
\end{equation}
Indeed, we assert that we have
\begin{equation}\label{h12}
\bar{y}_D + y_D = 0,
\end{equation}
for all positive divisors $D$ of $R$, which are not equal to $R$ and have $e_D \neq 0$. For such $D$, we have $y_D = 2v_D$.
But, for such $D$, we have $\bar{\psi}_D + \psi_D = T$, whence it follows that $2(\bar{v}_D + v_D) = 0$, because $K(\sqrt{D})$ is a
subfield of $\mathfrak H_D$. In view of \eqref{h12}, it follows that $\bar{u}_R + u_R = \bar{\psi}_R + \psi_R$, and \eqref{h10} follows easily.
Combining \eqref{h5} and \eqref{h10}, we conclude that $y$ must belong $E(\BQ(\sqrt{-\ell_0R}))^-$, and thus $y_R\in 2^{r-1}
E(\BQ(\sqrt{-\ell_0R}))^-+E(\BQ)[2]$. Suppose finally that
$y_R=2^ry'+t$, with $y'\in E(\BQ(\sqrt{-\ell_0R}))^-$ and $t\in
E(\BQ(\sqrt{-\ell_0R}))_\tor.$ If $m$ denotes an odd integer annihilating
the odd part of $E(\BQ(\sqrt{-\ell_0R}))_\tor$, we would then have
$$m\left(\psi_R-y'-\sum_{D|R, D\neq R}e_Dy_D\right)\in
E(\mathfrak H_R)[2^\infty]=E(\BQ)[2].$$ It follows that $m(\bar{\psi}_R+
\psi_R)=0$, contradicting \eqref{h8}, because $T$ is of order $2$. This completes the proof of
the theorem.

\section{Some classical 2-descents}

At present, we simply do not know how to prove the $2$-part of the
conjecture of Birch and Swinnerton-Dyer for the quadratic twists of
the elliptic curve $A=X_0(49)$, even though this should be the
easiest case to attack by the methods of Iwasawa theory, since every
such twist has complex multiplication by $F$ and has the prime $2$ as
a potentially good ordinary prime. Instead, we shall simply show
that a classical 2-descent argument in the spirit of Fermat (see,
for example, Chapter X, and in particular Prop. 4.9, of \cite{Sil})
establishes some partial results in this direction.

\medskip

In order to carry out the 2-descent, we must work with a new
equation for $A$ and its twists. Making the change of variables $x =
X/4 +2, y = Y/8 - X/8 -1$, we obtain the following equation for $A$:
$$
Y^2 = X^3 + 21X^2 + 112X.
$$
Let $M$ be any square free integer $\neq1$, and let $A^{(M)}$ be the
twist of $M$ by the quadratic extension $\BQ(\sqrt{M})/\BQ$. Then
the curve $A^{(M)}$ will have equation
$$
A^{(M)}: y^{2}=x^{3}+21Mx^{2}+112M^2x,
$$
and, dividing this curve by the subgroup generated by the
point $(0,0)$, we obtain the new curve
$$A^{'(M)}: y^{2}=x^{3}-42Mx^{2}-7M^2x.$$
Of course, $A^{'(M)}$ is the twist of $A' =A'^{(1)}$ by the quadratic extension $\BQ(\sqrt{M})/\BQ$. Explicitly, the isogenies between these two curves, are given by
$$
\phi: A^{(M)}\rightarrow A^{'(M)}, \  (x,y)\mapsto(\frac{y^{2}}{x^{2}},\frac{y(112M^{2}-x^{2})}{x^{2}})
$$
$$
\hat{\phi}: A^{'(M)}\rightarrow A^{(M)}, \  (x,y)\mapsto (\frac{y^{2}}{4x^{2}},\frac{y(-7M^{2}-x^{2})}{8x^{2}}).
$$
We write $S^{(\phi)}(A^{(M)})$ and
$S^{(\hat{\phi})}(A^{'(M)})$ for the classical Selmer groups of
the isogenies $\phi$ and $\hat{\phi}$, which can be described
explicitly as follows. Let $V$ denote the set of all places of
$\BQ$, and let $T_M$ be the set of primes dividing $14M$. Let
$\BQ(2, M)$ be the subgroup of $\BQ^\times/(\BQ^\times)^2$
consisting of all elements with a representative which has even
order at each prime number not in $T_M$. Writing
\begin{equation}\label{d1}
C_{d}: dw^{2}= 64 - 7\left(\frac{M}{d}z^2 + 3\right)^2,
\end{equation}
then $S^{(\phi)}(A^{(M)})$ can be naturally identified
with the subgroup of all $d$ in $\BQ(2,M)$ such that $C_d(\BQ_v)$ is
non-empty for $v=\infty$ and $v$ dividing $14M$. Similarly, writing
\begin{equation} \label{d2}
C'_{d} : dw^{2}=  1 + 7\left(\frac{2M}{d}z^2 + 3\right)^2
\end{equation}
then $S^{(\hat{\phi})}(A^{'(M)})$ can be naturally
identified with the subgroup of all $d$ in $\BQ(2,M)$ such that
${C'}_{d}(\BQ_v)$ is non-empty for $v=\infty$ and $v$ dividing
$14M$. Note that  $-7 \in S^{(\phi)}(A^{(M)})$ because it is the
image of the point $(0, 0)$ in $A^{'(M)}(\BQ)$, and similarly $7 \in
S^{(\hat{\phi})}(A^{'(M)})$ (see Proposition 4.9 of \cite{Sil}).
It will also be convenient for us to use the following
notation. If $D$ is any odd square free integer, we define $D_{+}$
(resp. $D_{-}$) to be the product of the primes dividing $D$, which
are $\equiv 1 \mod 4$ (resp. which are $\equiv 3 \mod 4$). In what
follows, we shall always assume that $M$ is prime to $7$, and we
will then write
\begin{equation}\label{con1}
M = \epsilon 2^\delta RN,
\end{equation}
where $\epsilon = \pm 1$, $\delta = 0, 1$, and $R$ (resp. $N$)
denotes the product of the prime factors of $M$ which are inert
(resp. split) in the field $F = \BQ(\sqrt{-7})$. To simplify the
statements of our results, we shall define a divisor $d$ of $M$ to
be {\it Confucian} if it satisfies the following condition at primes
$p$ dividing ${N_+}$:
\begin{equation}\label{con}
\left(\frac{d}{p}\right)=1 \, \textrm {when} \, p \, \textrm{divides} \,
N_+/(d, N_+), \, \textrm{and} \, \left(\frac{M/d}{p}\right)= \left(\frac{-7}{p}\right)_4
\, \textrm{when } \, p \,  \textrm{divides} \,  (d, N_+).
\end{equation}

\begin{prop}\label{descent1} Let $M$ be a square free integer prime to $7$.
Then $S^{(\phi)}(A^{(M)})$ consists of the classes in $\BQ(2,M)$
represented by all integers $d, -7d$ satisfying the following
conditions:
\begin{enumerate}
\item $d$ divides $2^\delta R_{-}N _{+}$.
\item When $M \equiv 1 \mod 4$, we have $d \equiv 1 \mod 4$, and when $M \equiv 3 \mod 4$, we have $d \equiv 1 \mod 8.$
\item When $M \equiv 6 \mod 8$, we have $d \equiv 1 \mod 8$, and when $M \equiv 2 \mod 8$, we have either $d \equiv 1 \mod 8$ or $d \equiv 5M \mod 16$.
\item We have $\left(\frac{d}{p}\right)=1$ for all primes $p$ dividing $N_-$.
\item $d$ is a Confucian divisor of $M$.
\end{enumerate}
\end{prop}
\begin{cor} \label{cd1}  Assume that $M$ is a square free integer, prime to $7$,
with $M\equiv 1\mod 4$. Then $S^{(\phi)}(A^{(M)})$ consists of the
classes in $\BQ(2, M)$ represented by integers $d, -7d$, where $d$
runs over all integers such that (i) $d\equiv 1\mod 4$, (ii) $d$
divides $R_-N_+$, (iii) $\left(\frac{d}{p}\right)=1$ for all primes $p$
dividing $N_-$, and (iv) $d$ is a Confucian divisor of $M$.
\end{cor}

\begin{proof} We recall that $C_d$ denotes the curve \eqref{d1}.
We see immediately that $C_d(\mathbb{R}) \neq \varnothing$, and that
$C_d(\BQ_7)\neq\varnothing $ if and only if
$\left(\frac{d}{7}\right)=1$. We now break up the rest of the
argument into a number of cases.

\medskip

Suppose that $q$ is any prime factor of $R$, and assume first that
$q$ divides $d$. Then we claim that $C_d(\BQ_q) \neq \varnothing$
if and only if $q$ divides $R_-$. Indeed, a point on $C_d$ with
coordinates in $\BQ_q$ must have coordinates in $\BZ_q$, whence it
follows easily that $\left(\frac{7}{q}\right) = 1$, or equivalently $q$ divides
$R_-$. Conversely if $q$ divides $R_-$, let $a$ be an integer such
that $a^2 \equiv 7 \mod q$. Then $(-3a+8)(-3a-8) \equiv -1 \mod q$ and
$-1$ is not a square modulo $q$. Thus one of $-3a+8$ and $-3a-8$ is a
square and one is a non-square modulo $q$. It follows that one of
the two congruences $\frac{aM}{d}z^2 \equiv (-3a \pm 8) \mod q$ must
always be soluble, and so $7(\frac{M}{d}z^2 + 3)^2 \equiv 64 \mod q$
is soluble, giving a point on $C_d$ with coordinates in $\BZ_q$. Now
assume that $q$ does not divide $d$. We will show that always
$C_d(\BQ_q) \neq \varnothing$. Indeed, if
$\left(\frac{d}{q}\right)=1$, the congruence $dw^2 \equiv 1 \mod q$
is soluble, giving a point on $C_d$ with coordinates in $\BZ_q$.
Otherwise $\left(\frac{-7d}{q}\right)=1$, and, taking $w=q^{-1}w_1,
z=q^{-1}z_1$, we obtain the equation $dw_1^2=64 q^2-7(\frac{M}{qd}
z_1^2+3q)^2$, which, on taking $z_1 = 1$ is plainly soluble modulo
$q$, yielding a point in $C_d(\BQ_q)$.

\medskip

Suppose next that $p$ is a prime dividing $N$, and assume first that
$p$ divides $d$. We claim that $C_d(\BQ_p)\neq \varnothing$ if and
only if $p \equiv 1 \mod 4$ and $\left(\frac{M/d}{p}\right)=\left(\frac{-7}{p}\right)_4$.
Indeed, we see easily that $C_d(\BQ_p) \neq \varnothing$ if and only
if $C_d$ has a point with coordinates in $\BZ_p$, and this will be
true if and only if the congruence given by looking at the equation
for $C_d$ modulo $p$ has a solution. It follows immediately that we
must have $(\frac{7}{p}) = 1$, whence we can assume that $p \equiv
1 \mod 4$ if $C_d(\BQ_p)\neq \varnothing$. We can therefore find
integers $e$ and $b$ such that $e^2 \equiv -7 \mod p$ and $b^2
\equiv -1 \mod p$.  We then have $(\frac{2b}{p})=1$, and $2b(3e-8b)
\equiv (3+eb)^2 \mod p$, whence $8b-3e$ is a square modulo $p$, and
so the same is true because $(8b + 3e)(8b - 3e) \equiv -1
\mod p$. Looking at the equation for $C_d$ modulo $p$, we conclude
easily that there will be a point modulo $p$ if and only if one of
the congruences
\begin{equation}\label{d3}
\frac{M}{d}z^2  \equiv -e(3e \pm 8b) \ \mod p
\end{equation}
is soluble. But since $-(3e \pm 8b)$ is a square modulo $p$, the
congruence \eqref{d3} will be soluble if and only if
$\left(\frac{M/d}{p}\right) = \left(\frac{e}{p}\right) = \left(\frac{-7}{p}\right)_4$,
proving our claim. Now assume that $p$ does not divide $d$. We claim that
$C_d(\BQ_p)\neq \varnothing$ if and only if $\left(\frac{d}{p}\right) = 1$.
Indeed, if $\left(\frac{d}{p}\right) = 1$, the congruence given by putting
$z=0$ in the equation of $C_d$ modulo $p$ is clearly soluble, and
this gives a point on $C_d$ with coordinates in $\BZ_p$. Conversely,
if there is a point on $C_d$ with coordinates in $\BZ_p$, it
follows immediately that $\left(\frac{d}{p}\right) = 1$. On the other hand, if
there is a point $(w, z)$ on $C_d$ with non-integral coordinates, we
can write $w=p^{-m}w_1, z=p^{-n}z_1$ with $m, n>0$ and $w_1, z_1\in
\BZ_p^\times$. It then follows that $m=2n-1$ and $dw_1^2=64
p^{2m}-7(\frac{M}{pd}z_1^2+3p^m)^2$. Taking this last equation
modulo $p$, we again conclude that $\left(\frac{d}{p}\right)=1$, as
required.

\medskip

Finally, we must determine when
\begin{equation}\label{d4}
C_d(\BQ_2) \neq \varnothing.
\end{equation}
We will show that \eqref{d4} holds for $d$ if and only if for $d$
even we have $d/2\equiv 1\mod 4$, and $M/d\equiv 5\mod 8$, and for
$d$ odd, we have either $d\equiv 1\mod 8$ or $d\equiv 5\mod 8$ and
$M/d\equiv 1\mod 4$. Suppose first that $2$ divides $d$. Then
\eqref{d4} will hold if and only if $M$ is even and there is a point
$(w, z)$ on $C_d$ with $w\in \BZ_2, z\in \BZ_2^\times$. It would
then follow that $\frac{M}{d} z^2+3\equiv 8\mod 16$,  and so the
equation  for $C_d$ becomes
$$
d\left(\frac{w}{8}\right)^2= 1-7\left(\frac{\frac{M}{d} z^2+3}{8}\right)^2 \equiv 2 \ \mod 8.
$$
If this congruence is soluble, then $d/2\equiv 1\mod 4$ and
$M/d\equiv 5\mod 8$. Conversely, Assume that $d/2\equiv 1\mod 4$,
and $M/d\equiv 5\mod 8$. Then the above congruence can be seen to be
soluble, whence also $C_d(\BQ_2)\neq \varnothing$, by noting that
$z^2\equiv c\mod 64$ is soluble for any $c\equiv 1\mod 8$. Next
suppose that $2$ does not divide $d$. If $d\equiv 1\mod 8$, we
claim that \eqref{d4} is always valid. Indeed, taking the equation
for $C_d$ modulo $8$, this congruence clearly has a solution if we
take $z=0$, showing that \eqref{d4} holds. Now assume that
$d\nequiv 1\mod 8$. We first show that there is no point on $C_d$
whose coordinates do not lie in $\BZ_2$. Suppose the contrary, and
put $w=2^{-m}w_1, z=2^{-n}z_1$ with $m, n>0$ and $w_1, z_1\in
\BZ_2^\times$, and let $M=2^\delta M_1$, so that $M_1$ is odd. Then
the equation for $C_d$ becomes
$$
2^{-2m}dw_1^2=2^6-7\left(\frac{M_1}{d}2^{\delta-2n}
z_1^2+3\right)^2.
$$
It follows that $m=2n-\delta$ and $d\equiv 1\mod 8$, which is a
contradiction. In order to investigate points on $C_d$ with
coordinates in $\BZ_2$, note that it is easy to see that $wz \neq
0$, and we put $w=2^m w_1, z=2^nz_1$ for some $w_1, z_1\in
\BZ_2^\times$, obtaining the equation becomes
$$2^{2m}dw_1^2=2^6-7\left(\frac{M_1}{d}2^{\delta+2n}
z_1^2+3\right)^2.$$ If $\delta+2n>0$, then $m=0$ and it follows that
$d\equiv 1\mod 8$, which is a contradiction. Thus we have
$\delta=n=0$ and our equation becomes
$$
dw_1^2=\frac{2^6-7\left(\frac{M}{d}z^2+3\right)^2}{2^{2m}}.
$$
If $M/d\equiv 3\mod 4$, then $m=1$ and $d\equiv 1\mod 8$, which is
again a contradiction. Thus we must have $M/d\equiv 1\mod 4$. If
$M/d\equiv 1\mod 8$, then necessarily $m=2$, and so we have
$$
dw_1^2=4-7\cdot 4^{-2}(\frac{M}{d}z^2+3)^2\equiv 5 \ \mod 8,
$$
which is a contradiction. If $M/d\equiv 5\mod 8$, then we can choose
$z$ such that $ord_2(\frac{M}{d}z^2+3)=4$ and $m=3$, and so we must have
$d\equiv 5\mod 8$. Conversely, if $d \equiv 5 \mod 8$ and $m=3$, it
follows that $C_d(\BQ_2) \neq \varnothing$. Putting together all of
the above results, the proof of Proposition \ref{descent1} is
complete.
\end{proof}

\medskip

\begin{prop}\label{descent2} Let $M$ be a square free integer prime to $7$. Then
$S^{(\hat{\phi})}(A^{'(M)})$ consists of all classes in $\BQ(2,M)$
represented by integers $d, 7d$ satisfying:
\begin{enumerate}
\item $d > 0$ and $d$ divides $2N$.
\item  If $M \equiv 1 \mod 4$, then $d$ is odd, and if $M\equiv 2 \mod 8$,
we have either $d\equiv \pm 1 \mod 8$ or $d \equiv \pm 3M \mod 16$.
\item $(\frac{d}{q})=1$ for all primes $q$ dividing $R_-$.
\item $d$ is a Confucian divisor of $M$.
\end{enumerate}
\end{prop}
\begin{cor} \label{cd2} Assume that $M$ is a square free integer prime to $7$, with
$M\equiv 1\mod 4$. Then $S^{(\hat{\phi})}(A^{'(M)})$ consists of the
classes in $\BQ(2,M)$ represented by integers $d, 7d$ where $d$ runs
over all integers satisfying (i) $d > 0$,  (ii) $d$ divides $N$, (iii)
$\left(\frac{d}{q}\right)=1$ for all primes $q$ dividing $R_-$, and
(iv) $d$ is a Confucian divisor of $M$.
\end{cor}
\begin{proof} We recall that $C'_d$ denotes the curve \eqref{d2}.
It is clear that $C'_d(\BR)\neq \varnothing$ if and only if $d>0$,
and that $C'_d(\BQ_7)\neq \varnothing$ if and only if
$\left(\frac{d}{7}\right)=1$.

\medskip

Suppose first that prime $p$ is an odd prime dividing both $d$ and
$M$. Then
\begin{equation}\label{d5}
C'_d(\BQ_p) \neq \varnothing
\end{equation}
if and only if there is a point on $C'_d$ with $w, z$ in $\BZ_p$,
and this latter assertion will hold if and only if the congruence
given by the equation of $C'_d$ modulo $p$ is soluble. In
particular, it follows that a necessary condition for \eqref{d5} to
hold is that $\left(\frac{-7}{p}\right)=1$, which is equivalent to
saying that $p$ is a divisor of $N$. We can assume therefore that
$p$ divides $N$.  As in the proof of the previous proposition, let
$e$ be an integer such that $e^2 \equiv - 7 \mod p$. Then it is easily
seen that the congruence given by the equation of $C'_d$ modulo $p$
will be soluble if and only if the congruence
$$
\frac{M}{d} (ez)^2\equiv e(-6e\pm 2) \ \mod p
$$
is soluble. Note that $(-3e+1)(-3e-1) \equiv -64 \mod p$ and that $(-6e +
2) \equiv (3 - e)^2 \mod p$. Hence \eqref{d5} will always be true
when $p$ divides $N_-$, and it will be true when $p$ divides $N_+$
if and only if
$\left(\frac{M}{d}\right)=\left(\frac{-7}{p}\right)_4$. Suppose next
that $p$ is an odd prime dividing $M$ which does not divide $d$. We
claim that $\eqref{d5}$ is valid if and only if $\left(\frac{d}{p}\right) = 1$
or $\left(\frac{7d}{p}\right) = 1$. Indeed, it is clear that there will be a
point on $C'_d$ with coordinates in $\BZ_p$ if and only if
$\left(\frac{d}{p}\right) = 1.$ For non-integral points on $C'_d$, put
$w=p^{-m}w_1, z=p^{-n} z_1$ with $w_1, z_1 \in \BZ_p^\times$ and $m,
n>0$. Then we must have $m=2n-1$,  and we obtain the new equation
$$
dw_1^2=p^{2m}+7\left(\frac{2M}{pd}z_1^2+3 p^m\right)^2.
$$
This equation is clearly soluble modulo $p$, if and only if
$\left(\frac{7d}{p}\right)=1$. Since $(d, R) =1$, it follows that
\eqref{d5} is true for primes $p$ dividing $R_+$, and \eqref{d5}
holds for primes $p$ dividing $R_-$ if and only if $\left(\frac{d}{p}\right) =
1$. Similarly for primes $p$ dividing $N/(N,d)$, \eqref{d5} is valid
for all primes $p$ dividing $N_-$, and for primes $p$ dividing $N_+$
if and only if $\left(\frac{d}{p}\right) = 1$.

\medskip

We now determine when
\begin{equation}\label{d6}
C'_d(\BQ_2) \neq \varnothing.
\end{equation}
We assume first that $M$ is odd. Under this assumption, we will
prove that \eqref{d6} is always true when $d$ is odd, and that it is
true for $d$ even if and only if $M\equiv 3 \mod 4$. Suppose first
that $2$ divides $d$, say $d=2d_1$.  It is then clear that
\eqref{d6} is valid if and only if there is a point with coordinates
in $\BZ_2$, and one can verify without too much difficulty that both
coordinates of such a point should be non-zero.  Write $w=2^m w_1,
z=2^nz_1$ with $w_1, z_1\in \BZ_2^\times$ and $m, n\geq 0$, so that
$$
d_1w_1^2=\frac{2^6+ 2^{2n+1}\cdot 21(Mz_1^2/d_1)+ 2^{4n}\cdot
7(Mz_1^2/d_1)^2}{2^{2m+1}}.
$$
If this equation is soluble, we must have $n\neq 0$ and $2m+1=2n+1<6$,
whence $n=m=1, 2$. If $n=m=1$, then
$$
d_1w_1^2=2^3+21(Mz_1^2/d_1)+2\cdot
7(Mz_1^2/d_1)^2\equiv 5M/d_1-2 \ \mod 8,
$$
which is soluble if and only if $5M-2d_1\equiv 1 \mod 8$. If
$n=m=2$, then
$$
d_1w_1^2=2+21 (Mz_1^2/d_1)+2^3\cdot 7 (Mz_1^2/d_1)^2\equiv
2+5M/d_1 \ \mod 8,
$$
which is soluble if and only if $5M + 2d_1 \equiv 1 \mod 8$.
Hence we see that \eqref{d6} can hold only if $M \equiv 3 \mod 4$.
Conversely, if $M \equiv 3 \mod 4$, one verifies
easily that either $5M + 2d_1$ or $5M - 2d_1$ must be $\equiv 1 \mod
8$, proving our claim for the case when $2$ divides $d$. Assume now that
$d$ as well as $M$ is odd, and we will proceed to show that
\eqref{d6} is always true. This is easily seen to be true if
$d\equiv 1\mod 8$. Writing $w=2^m w_1, z=2^n z_1$ with $m, n\geq
0$ and $w_1, z_1\in \BZ_2^\times$, we have the new equation
$$
dw_1^2=\frac{2^6+ 2^{2n+2}\cdot 21(Mz_1^2/d)+ 2^{4n+2}\cdot
7(Mz_1^2/d)^2}{2^{2m}}.
$$
If $d\equiv 5\mod 8$, take $n=m=3$, then
$$
dw_1^2=1+ 2^2\cdot 21(Mz_1^2/d_1)+2^8\cdot 7(Mz_1^2/d_1)^2\equiv 5 \ \mod 8,
$$
from which we see \eqref{d6} also holds in this case. Let us now
assume that $d\equiv 3\mod 4$ and $M \equiv 1 \mod 4$. Take $n=1$,
then $m=2$ and we have
$$
dw_1^2=2^2(1+7(Mz_1^2/d)^2)+ 21(Mz_1^2/d)\equiv 5M/d \ \mod 8,
$$
which is soluble when $M\equiv 5\mod 8$. Assume now that
$M\equiv 1\mod 8$. Taking $n=2$, our equation becomes
$$
dw_1^2=\frac{2^{-2}\left(1+ 21(Mz_1^2/d)\right)+2^2\cdot 7(Mz_1^2/d)^2}{2^{2m-8}}.
$$
If $d\equiv 3\mod 8$, we take $m=5$ and $z_1$ such that
$1+21Mz_1^2/d\equiv 64\mod 128$; if $d\equiv 7\mod 8$, we take $m=4$
and $z_1$ such that $1+21Mz_1^2/d\equiv 12 \mod 32$. Note that for
any $a\equiv 1\mod 8$, we may find $x\in \BQ_2$ such that $x^2=a$
and therefore $x^2\equiv a\mod 2^n$ is soluble in $\BZ$ for any $n
\geq 3$. Thus we are finally left with the case when $d\equiv 3 \mod
4$ and $M \equiv 3 \mod 4$. In the equation for $C'_d$, put $w =
2^{-m}w_1, z = 2^{-n}z_1$, with $m=2n-1 > 0$, and $w_1, z_1$ in
$\BZ_2^\times$. Then the equation of $C_d'$ becomes
$$
dw_1^2=2^{2m}+7\left(\frac{M}{d}z_1^2+3\cdot 2^m\right)^2.
$$
If $d\equiv 7\mod 8$, we take $m\geq 2$ and $z_1=1$, and \eqref{d6}
is valid. If $d\equiv 3\mod 8$, take $m=1$, so that we have
$dw_1^2=4+7(\frac{M}{d}z_1^2+6)^2$. Take $w_1$ such that
$4-dw_1^2=b^2$ with $b=1$ or $3$, and let $u^2=-7$ so that
$u\equiv \pm 3\mod 8$. Now we have equation
$$
(uMz_1/d)^2=\frac{M}{d}(-6u^2\pm ub)\equiv 3M(2\pm 3b)\equiv M(-2\pm b)\mod 8.
$$
If $M\equiv 3\mod 8$, take $\pm b=-3$, and if $M\equiv 7\mod 8$, take
$\pm b=1$, and it follows that \eqref{d6} holds in this case,
completing the proof of all of our claims about \eqref{d6} when $M$
is odd.

\medskip

Assume now that $M$ is even, and write $M=2M_1$. Replacing $z$ by
$2z$, the equation for $C'_d$ then becomes
\begin{equation}\label{d7}
dw^2=1+7\left(\frac{M_1}{d}z^2+3\right)^2.
\end{equation}
Suppose first that $d$ is odd. We shall prove that \eqref{d6} is
true if and only if either $d\equiv \pm 1 \mod 8$, or $d \equiv \pm
3 \mod 8$ and $M_1 \equiv 3 \mod 4$. Note first that \eqref{d7} has
a solution in $\BQ_2$ with $z=0$ if and only if $d \equiv 1 \mod 8$,
and there is no solution when $w=0$.  Put $w=2^{-m}w_1, z =
2^{-n}z_1$, where $m, n > 0$, and  $w_1, z_1$ are in $\BZ_2^\times$.
Then a necessary condition for a solution is that $m = 2n$, and we
then obtain the new equation
$$
dw_1^2=2^{4n}+7\left(\frac{M_1}{d}z_1^2+3\cdot 2^{2n}\right)^2.
$$
Since the right hand side of this equation is $\equiv 7 \mod 8$, we
deduce that it has a solution if and only if $d \equiv 7 \mod 8.$
Next assume that $w=2^{m}w_1, z = 2^{n}z_1$, where $m, n \geq 0$, and
$w_1, z_1$ are in $\BZ_2^\times$, so that the equation becomes
\begin{equation}\label{d8}
dw_1^2=\frac{2^6+2^{2n+1} \cdot 21(M_1z_1^2/d)+2^{4n}\cdot
7(M_1z_1^2/d)^2}{2^{2m}}.
\end{equation}
It is easily seen that the solubility of \eqref{d8} implies that
$n=m=0$ or $m=3$ and $n\geq 3$. In the case $n=m=0$, \eqref{d8} will
be soluble if and only if $2M_1-d \equiv 1 \mod 8$; In the case
$n=m=3$, it will be soluble if and only if
 $2M_1+d \equiv 1 \mod 8$. In the case
$m=3$ and $n\geq 4$, \eqref{d8} is soluble if and only if $d \equiv
1\mod 8$. This proves our claim about the validity of \eqref{d6} in
this case.

\medskip

Now assume finally that $d$, as well as $M$, is even, and put
$d=2d_1$. We will prove that, in this case,  $\eqref{d6}$ is valid
if and only if $M_1\nequiv 1\mod 8$ when $d_1 \equiv \pm 1 \mod 8$,
and if and only if $M_1 \nequiv 5 \mod 8$ when $d_1 \equiv \pm 3
\mod 8.$ It is easy to see that all points on \eqref{d7} with
coordinates in $\BQ_2$ must have  $w$ and  $z$ in $\BZ_2\setminus
\{0\}$ and $z$ in $2\BZ_2$. Put $w=2^mw_1, z=2^n z_1$ with $w_1,
z_1\in \BZ_2^\times$ and $m\geq 0, n\geq 1$, and we can then rewrite
\eqref{d7} as
$$
d_1w_1^2=\frac{2^6+ 2^{2n}\cdot 21(M_1z_1^2/d_1)+2^{4n-2}\cdot
7(M_1z_1^2/d_1)^2}{2^{2m+1}}.
$$
It follows easily that $n=1$ or $3$. If $n=1$, this equation
becomes
\begin{equation}\label{d9}
d_1w_1^2=\frac{7(M_1z_1^2/d_1)\left((M_1z_1^2/d_1)+3\right)+2^4}{2^{2m-1}}.
\end{equation}
For this equation, a necessary condition for a solution is that
$v=ord_2(M_1z_1^2/d_1+3)$ must be equal to $1, 3$, or $4$, and we
analyse solubility in each of these cases. If $v=1$, we have $m=1$
and thus
$$
(d_1w_1)^2=7M_1z_1^2\cdot
\frac{\frac{M_1}{d_1}z_1^2+3}{2}+2^3d_1\equiv
-\frac{(M_1z_1)^2+3M_1d_1}{2d_1} \ \mod 8,
$$
which is soluble if and
only if $-2d_1\equiv x^2+3M_1d_1\mod 16$ is soluble, and this will
happen if $3M_1+d_1+2\equiv 0\mod 8$. Suppose next that $v=3$, so
that we then have $m=2$ and
$$
(d_1w_1)^2=7M_1z_1^2\cdot
\frac{\frac{M_1}{d_1}z_1^2+3}{8}+2d_1\equiv
-\frac{(M_1z_1)^2+3M_1d_1}{8d_1}+2d_1 \ \mod 8.
$$
One then verifies that this
last equation is soluble if and only if $d_1\equiv 5M_1\mod 8$. If
$v=4$ one also sees easily that \eqref{d9} is soluble only if
$M_1d_1\equiv 5\mod 8$, in which case the previous case when $v=2$
shows that there is a solution of \eqref{d9}.
Finally, we have to consider the case $n=3$, when the equation
becomes
\begin{equation}\label{d10}
d_1w_1^2=\frac{\left(1+21M_1z_1^2/d_1\right)+2^4\cdot
7(M_1z_1^2/d_1)^2}{2^{2m-5}}.
\end{equation}
It is then easy to see that $u=ord_2(1+21 M_1z_1^2/d_1)$ must be equal to $1, 3$, or $4$.
When $u=1$, which is equivalent to $d_1+M_1\equiv 2\mod 4$,  we have $m=3$ and
$$
(d_1w_1)^2=\frac{d_1+21M_1z_1^2}{2}+2^3\cdot
7d_1(M_1z_1^2/d_1)^2\equiv \frac{d_1+5M_1z_1^2}{2}\mod 8
$$
which is soluble if and only if $5M_1+d_1-2\equiv 0\mod 8$. When
$u=3$, which is equivalent to $1+21M_1z_1^2/d_1\equiv
8\mod 16$, we conclude that necessarily $d_1\equiv 3M_1\mod 8$. Thus
we must have $m=4$ and
$$(d_1w_1)^2=\frac{d_1+21M_1z_1^2}{8}+2\cdot 7d_1(M_1z_1^2/d_1)^2
\equiv \frac{d_1+5M_1z_1^2}{8}+ 6d_1\mod 8,$$ which is soluble
if and only if $d_1\equiv 3M_1\mod 8$. When $u=4$, we have
$1+21M_1z_1^2/d_1\equiv 2^4\mod 2^5$, which implies $d_1\equiv
3M_1\mod 8$, and thus there is already solubility from the previous
case when $u=3$. This completes the proof of the analysis of when
\eqref{d6} is valid, and the assertions of the proposition now
follow by putting together the cases discussed above.

\end{proof}

We now give some consequences of Propositions \ref{descent1} and
\ref{descent2}. We will assume for the whole of this paragraph that $M$ is
a square free integer, prime to $7$, with $M \equiv 1\, \mod \, 4$, and, as in \eqref{con1}, we write $M = \epsilon RN$. In
particular, it follows that the curve $A^{(M)}$ always has good
reduction at 2, and its $L$-function has global root number $+1$
(reps. $-1$) when $M > 0$ (resp. $M < 0$). We write
$S^{(2)}(A^{(M)})$ for the classical Selmer group of $A^{(M)}$ for
the endomorphism given by multiplication by $2$. Now it is easily
seen that we have an exact sequence
\begin{equation}
0 \to A'^{(M)}(\BQ)[\hat{\phi}] \to S^{(\phi)}(A^{(M)}) \to
S^{(2)}(A^{(M)}) \to S^{(\hat{\phi})}(A'^{(M)}).
\end{equation}
Define $\mathfrak S^{(\phi)}(A^{(M)})$ and $\mathfrak S^{(2)}(A^{(M)})$ to
be the quotients of $S^{(\phi)}(A^{(M)})$ and $S^{(2)}(A^{(M)})$ by
the images of the torsion subgroups of $A'^{(M)}(\BQ)$ and
$A^{(M)}(\BQ)$, respectively. Since $A^{(M)}$ and $A'^{(M)}$ have
good reduction at $2$, the theory of complex multiplication shows that
the $2$-primary subgroups of
$A'^{(M)}(\BQ)$ and $A^{(M)}(\BQ)$ are both just of order $2$, whence it follows easily that we have the exact sequence
\begin{equation}\label{d11}
0 \to \mathfrak S^{(\phi)}(A^{(M)}) \to \mathfrak S^{(2)}(A^{(M)}) \to S^{(\hat{\phi})}(A'^{(M)}).
\end{equation}
Note also that the parity theorem of the Dokchitser brothers \cite{D} shows that $\mathfrak S^{(2)}(A^{(M)})$ has even or odd $\mathbb{F}_2$-dimension according as $M > 0$ or $M < 0$.

\begin{cor}\label{descent3} Assume that $M = R_+$. Then $\mathfrak S^{(2)}(A^{(M)}) = 0.$
\end{cor}
\begin{proof}
Indeed Proposition \ref{descent1} shows that, in this case, we have $\mathfrak S^{(\phi)}(A^{(M)})=0$, and Proposition \ref{descent2} shows that $S^{(\hat{\phi})}(A'^{(M)})$ has order $2$, whence the assertion follows from the exact sequence \eqref{d11}, and the fact that $\mathfrak S^{(2)}(A^{(M)})$ must have even $\mathbb{F}_2$-dimension.
\end{proof}
\begin{cor}\label{descent4} Assume that $M = R$ with $M \equiv 1\, \mod \, 4$, and let $r_-(M)$ denote the number of prime factors of $R_-$. Then  $\mathfrak S^{(2)}(A^{(M)})$ has exact order  equal to $2^{r_-(M)}$.
\end{cor}
\begin{proof} One simply invokes Corollaries \ref{cd1} and \ref{cd2} and the exact sequence \eqref{d11}, and notes that $\mathfrak S^{(2)}(A^{(R)})$ must have even dimension
over $\mathbb{F}_2$ because the root number of $A^{(R)}$ is $+1$.
\end{proof}
\begin{cor}\label{descent5} Assume that $M = R_+N_-$, with $M \equiv 1\, \mod \, 4$,  and let $k_-(M)$ denote the number of prime factors of $N_-$. Then $\mathfrak S^{(\hat{\phi})}(A'^{(M)})$ has exact order $2^{k_-(M)}$, and $\mathfrak S^{(2)}(A'^{(M)})$ has order at least equal to $2^{k_-(M)}$.
\end{cor}
\begin{proof} The first assertion is clear from Proposition \ref{descent2}, and so the corollary is clear from the analogue of the exact sequence \eqref{d11} for $A'^{(M)}$.
\end{proof}
\begin{cor}\label{descent6} Assume that $M = RN_+$, where $M \equiv 1\, \mod \, 4$ and $N_+ > 1$. Assume further that every prime factor of $N_+$ splits completely in the field $\BQ(i, \sqrt[4]{-7},\sqrt{R})$. Then $\mathfrak S^{(2)}(A^{(M)}) \neq 0$.
\end{cor}
\begin{proof} We claim that, under the hypotheses of the corollary, we always have $N_+$ in $\mathfrak S^{(2)}(A^{(M)})$. But this is clear from Proposition \ref{descent1} because $N_+ \equiv 1 \mod 4$, and $N_+$ is indeed a Confucian divisor of $M$, thanks to our hypothesis that every prime factor of $N_+$ splits completely in the field $\BQ(i, \sqrt[4]{-7})$.
\end{proof}
\begin{cor}\label{descent7} Assume that $M = - \ell_0R_+N_+$, where $l_0$ is a prime such that $\ell_0 \equiv 3 \mod 4$ and $\ell_0$ is inert in $F$. Assume further that every prime factor of $N_+$
splits completely in the field obtained by adjoining to $\BQ(i,
\sqrt[4]{-7})$ the square roots of all primes dividing $R_+$. Then
$\mathfrak S^{(2)}(A^{(M)})$ has order $2$ if and only if the ideal class
group of the imaginary quadratic field $\BQ(\sqrt{-\ell_0N_+})$ has no
element of exact order $4$. \end{cor}
\begin{proof} Put $V = -\ell_0N_+$. Under the hypotheses of the corollary, it is clear that $V$ is always a Confucian divisor of $M$, which is $\equiv 1 \mod 4$, and so belongs to $\mathfrak S^{(\phi)}(A^{(M)})$. Moreover, Propositions
\ref{descent1} and \ref{descent2} show that we then have
\begin{equation}\label{descent13}
 S^{(\phi)}(A^{(M)})= \{1, -7, V, -7V\}, \, \, \, \textrm{and} \, \, \, S^{(\hat{\phi})}(A'^{(M)}) = \{1, 7\},
\end{equation}
or equivalently that $\mathfrak S^{(2)}(A^{(M)})$ has order $2$,  if and only if there is no Confucian divisor $d\neq 1, V$ of $V$ with $d\equiv 1 \mod 4$. But the classical theory of genera, interpreted via the R\'{e}dei matrix, shows that this last assertion holds if and only if the imaginary quadratic field $\BQ(\sqrt{V})$ has no element of order $4$ in its ideal class group.
 \end{proof}

In order to be able to compare the $2$-descent arguments given above with the predictions of the $2$-part of the conjecture of Birch and Swinnerton-Dyer, we give the Tamagawa factors for the curves $A^{(M)}$ and $A'^{(M)}$, with a brief indication of proofs. We assume once again that $M$ is an arbitrary square free integer prime to 7, and write $D_M$ for the discriminant of the field $\BQ(\sqrt{M})$. Note that both $A^{(M)}$ and $A'^{(M)}$ have bad additive reduction at all primes dividing $7D_M$. Write $c_p(A^{(M)})$ for the Tamagawa factor of $A^{(M)}$ at a finite prime $p$, and similarly for $A'^{(M)}$. If $p$ is an odd prime of bad additive reduction, it is well known (see \cite{C2}, Lemmas 36 and 37) that
\begin{equation}\label{descent10}
c_p(A^{(M)}) = \#(A(\BQ_p)[2]), \, \, \, c_p(A'^{(M)}) = \#(A'(\BQ_p)[2]).
\end{equation}
Also, writing $F = \BQ(\sqrt{-7})$ and $F' = \BQ(\sqrt{7})$, we have
\begin{equation}\label{descent11}
F = \BQ(A[2]), \, \, \, \, F' = \BQ(A'[2]).
\end{equation}
\begin{prop}\label{descent8} For all square free integers $M$, we have (i) $A^{(M)}(\BR)$ has one connected component, (ii)
$c_2(A^{(M)})$ is equal to $1$ or $4$, according as $D_M$ is odd or even, (iii) $c_7(A^{(M)}) = 2$,
(iv) $c_p(A^{(M)}) = 2$ if $p$ is an odd prime dividing $M$, which is inert in $F$, and (v) $c_p(A^{(M)}) = 4$ if $p$ is an odd prime
dividing $M$, which is split in $F$.
\end{prop}
\begin{proof} Assertion (i) follows immediately from the fact that $F= \BQ(A[2])$. For assertion (ii), one has to use Tate's algorithm for computing
the Tamagawa factor $c_2(A^{(M)})$ when $D_M$ is even. The remaining assertions involving odd primes $p$ of bad reduction follow immediately from
\eqref{descent10}, on noting that $A(\BQ_p)[2]$ is of order $2$ or $4$, according as $p$ is not or is split in $F$.
\end{proof}
\begin{prop}\label{descent9} For all square free integers $M$, we have (i)
 $A'^{(M)}(\BR)$ has two connected components, (ii) $c_2(A'^{(M)})$ is equal to $1$ if $D_M$ is odd,
 to $2$ if $4$ exactly divides $D_M$, to $2$ if $8$ divides $D_M$ and $M/2 \equiv 3 \mod 4$, and to $4$ if $8$ divides $D_M$ and $M/2 \equiv 1 \mod 4$, (iii) $c_7(A'^{(M)})=2$, (iv) if $p$ is an odd prime dividing $M$, which is inert in $F$, then $c_p(A'^{(M)})$ is equal to $2$ or $4$ according as $p \equiv 1 \mod 4$ or $p \equiv 3 \mod 4$, and (v) if $p$ is an odd prime dividing $M$, which splits in $F$, then $c_p(A'^{(M)})$ is equal to $2$ or $4$ according as $p \equiv 3 \mod 4$ or $p \equiv 1 \mod 4$.
\end{prop}

\begin{proof} Assertion (i) follows immediately from the fact that $F' = \BQ(A'[2])$. To establish assertion (ii) when $D_M$ is even, one has to use Tate's algorithm for computing Tamagawa factors. The remaining assertions involving odd primes of bad reduction follow immediately from \eqref{descent10}, on noting that $A'(\BQ_p)[2]$ is of order $2$ or $4$, according as $p$ is not or is split in $F'$.
\end{proof}

We recall that the Birch-Swinnerton-Dyer conjecture is known to be compatible with isogenies, and we now determine the
explicit relationship between the orders of the Tate-Shafarevich groups of
$A^{(M)}$ and $A'^{(M)}$, which follows from this compatibility. We again assume that $M$ is any square free integer, which is written in the form \eqref{con1}. As earlier, we write $r_-(M)$ for the number of prime factors of $R_-$, and $k_-(M)$ for the number of prime factors of $N_-$.
Define $g(M)=\rank_\BZ A^{(M)}(\BQ)=\rank_{\BZ} A'^{(M)}(\BQ).$ Let $\rho(M) \leq g(M)$ be defined by
$$\rho(M) = \textrm{ord}_2([A^{'(M)}(\BQ):
\phi(A^{(M)}(\BQ))+A^{'(M)}(\BQ)_\tor]).$$
Finally, define $a(M)$ to be $1$ if $N_-R_-\equiv
-\sign (M)\mod 4$,  and $0$ if $N_-R_-\equiv
\sign (M)\mod 4$. Write $\Sha(A^{(M)})$ and $\Sha(A'^{(M)})$ for the Tate-Shafarevich groups of $A^{(M)}$ and $A'^{(M)}$, viewed as elliptic curves over $\BQ$.
\begin{prop} Let $M$ be a square-free integer prime to $7$. Then
$\Sha(A^{(M)})$ and $\Sha(A'^{(M)})$ are either both infinite or both finite, and in the latter case we have
$$\frac{\#(\Sha(A'^{(M)}))}{\#(\Sha(A^{(M)}))}=2^{a(M)+k_-(M) - r_-(M)+2\rho(M)-g(M)}.$$
\end{prop}
\begin{proof} Let $\Omega_\infty(A^{(M)})$ (resp.  $\Omega_\infty(A'^{(M)})$) be the integral of a \Neron differential over the whole of $A^{(M)}(\BR)$ (resp. $A'^{(M)}(\BR)$).
Define
$$
Tam(A^{(M)}) = \prod_p c_p(A^{(M)}), \, \, \, \, Tam(A'^{(M)}) = \prod_p c_p(A'^{(M)}),
$$
where the products are taken over all primes of bad reduction. Now the finiteness of one of the Tate-Shafarevich groups implies the finiteness of the other one.
Assuming both are finite, the invariance of the Birch-Swinnerton-Dyer conjecture under isogeny gives
$$
\frac{Tam(A^{(M)})\cdot \Omega_\infty(A^{(M)})\cdot R(A^{(M)})\cdot
\#(\Sha(A^{(M)}))}{\#(A^{(M)}(\BQ)_\tor)^2}=\frac{Tam(A'^{(M)})\cdot \Omega_\infty(A^{'(M)})\cdot R(A^{'(M)})\cdot \#(\Sha(A'^{(M)}))}{\#(A'^{(M)}(\BQ)_\tor)^2}.
$$
Here the regulator terms $R(A^{'(M)})$ and $R(A^{(M)})$
are volumes with respect to the \Neron-Tate pairing, and we have
$$\begin{aligned}
R(A'^{(M)})&=\Vol(A'^{(M)}(\BQ))^2=2^{-2\rho(M)}
\Vol(\varphi(A^{(M)}(\BQ)))^2\\
&=
2^{-2\rho(M)+g(M)}\Vol(A^{(M)}(\BQ))^2=2^{-2\rho(M)+g(M)}R(A^{(M)}).\end{aligned}$$
Thus
$$R(A^{(M)})/R(A'^{(M)})=2^{2\rho(M)-g(M)}.$$
It also follows easily from Propositions \ref{descent8} and \ref{descent9} that
$$\frac{Tam(A^{(M)})}{Tam(A'^{(M)})}=2^{a(M)+ k_-(M) - r_-(M)}.$$
Also, we have
 $$A^{(M)}(\BQ)_\tor=A'^{(M)}(\BQ)_\tor\cong \BZ/2\BZ.$$
Moreover, let us define $\omega(A)$ (resp. $\omega^-(A)$) and  $\omega(A')$
(resp. $\omega^{-}(A')$) to be the least positive real period (resp. purely imaginary period in the upper half plane) of
the \Neron differentials on $A$ (resp. on $A'$). We then have the following period relations.  Define
$u=1$ if $M\equiv 1\mod 4$, $u=1/2$ if otherwise. Then it is not difficult to see (see \cite{Pal}) that
$$\Omega_\infty(A^{(M)})=\begin{cases} \frac{u}{\sqrt{M}} \omega(A),
&\text{if $M>0$},\\
\frac{u}{\sqrt{M}}\omega^-(A),\quad& \text{if $M<0$}.\end{cases}
\qquad\qquad
\Omega_\infty(A^{'(M)})=\begin{cases} \frac{u}{\sqrt{M}} 2\omega(A'),\quad &\text{if $M>0$},\\
\frac{u}{\sqrt{M}}2\omega^-(A'), &\text{if $M<0$}.\end{cases}$$ But
it is also easy to see that $\omega(A)= 2\omega(A')$, and
$\omega^-(A)=2\omega^-(A')$, whence
$$\Omega_\infty(A^{(M)})/\Omega_\infty(A'^{(M)})=1,$$
and the proof is complete.

\end{proof}

\section{Zhao's Method}

The aim of this section is to show that one can use Zhao's method
(see, for example, \cite{CKLZ}, where one can also find references
to his earlier papers) to establish a few of the analytic results
which would follow from the descent calculations of the previous
section if we knew (but we do not know) the $2$-part of the
conjecture of Birch and Swinnerton-Dyer. Specifically, we prove
Theorem \ref{main3} and the analytic part of Theorem \ref{main4}.
Throughout this section, $M$ will denote a square free element of
the ring of integers $\CO$ of $F$ (it will not, in general, be a
rational integer), which we will always assume satisfies $M \equiv 1
\mod 4\CO.$

 \medskip

We begin by establishing some preliminary results, which will be
needed for the proof of the above theorems. Recall that $\CO$
denotes the ring of integers of the field $F = \BQ(\sqrt{-7})$. We now view our elliptic curve $A = X_0(49)$ as being defined over $F$. If $M$ is a square
free element of $\CO$, which is $\equiv 1 \mod \ 4$, we write
$\psi_M$ for the Grossencharacter of $A$ twisted by the quadratic
extension $F(\sqrt{M})/F.$  As in \cite{CKLZ}, let $\pi_1, \ldots,
\pi_m$ be an arbitrary sequence of distinct prime elements of $\CO$
such that, for all $m \geq 1$, we have (i) $(\pi_m, \sqrt{-7})=1$,
and (ii) $\pi_m \equiv 1 \mod 4$. Recall that the period lattice
of the \Neron differential on our minimal Weierstrass equation for
$A$ is given by $\mathfrak L = \Omega_\infty \CO$.  For all $m \geq 1$,
define
$$
\mathfrak M_m = \pi_1\cdots\pi_m, \, \, g_m = \mathfrak M_m \sqrt{-7}, \, \,
\mathfrak g_m = g_m\CO,
$$
and let $\mathfrak R_m$ be the ray class field of $F$ modulo $\mathfrak
g_m$, which coincides with the field $F(A_{g_m})$ (see Lemma 7 of
\cite{C}). Then, as is explained in \cite{CKLZ},  the field $\mathfrak
J_m = F(\sqrt{\pi_1}, \ldots, \sqrt{\pi_m})$ is a subfield of $\mathfrak
R_m$. As in \cite{CKLZ}, let ${\mathcal E}_1^*(z, \mathfrak L)$ be the
non-holomorphic Eisenstein series of weight 1 for the lattice $\mathfrak
L$. The following result strengthens Theorem 3.1 of \cite{CKLZ}. For all $m \geq 1$, we define
$$
\Psi_m = Tr_{\mathfrak R_m/\mathfrak J_m}(g_m^{-1}{\mathcal E}_1^*(\Omega_\infty/g_m, \mathfrak L)).
$$
\begin{prop}\label{i}
For all $m\geq 1$, $\Psi_m$ is integral at all places of $\mathfrak J_m$ above $2$.
\end{prop}
We now give the proof of Proposition \ref{i}, beginning
with two preliminary lemmas. Put $f = \sqrt{-7}$, \, $\mathfrak f =
f\CO$, and define $\mathfrak F = F(A_f)$. Since $\mathfrak f$ is the
conductor of the Grossencharacter of $A$, $\mathfrak F$ coincides with
the ray class group of $F$ modulo $\mathfrak f$ (see Lemma 7 of
\cite{C}), and so is an extension of $F$ of degree $3$.  Moreover, the
action of the Galois group of $\mathfrak R_m$  over $\mathfrak F$ on
$A_{\mathfrak M_m}$ gives rise to an injection
\begin{equation}\label{ir}
j: Gal(\mathfrak R_m/\mathfrak F) \to Aut_\CO(A_{\mathfrak M_m}) = (\CO/\mathfrak
M_m\CO)^{\times}.
\end{equation}
\begin{lem}\label{irl} The homomorphism $j$ given by \eqref{ir} is an isomorphism.
\end{lem}
\begin{proof} Since $A$ has good reduction at $\mathfrak p_j = \pi_j\CO$, the formal group of $A$ at $\mathfrak p_j$
is a Lubin-Tate group. It follows that the Galois group of the
extension $F(A_{\mathfrak p_j})/F$ is isomorphic to $(\CO/\mathfrak
p_j)^{\times}$, and that $\mathfrak p_j$ is totally ramified in this
extension. Also, $\mathfrak p_j$ does not ramify in the extension $\mathfrak
F/F$. The assertion of the lemma now follows easily.\end{proof}

\medskip

Since $(f, \mathfrak M_m)=1$, we can find $\alpha, \beta$ in $\CO$ such
that $1 = \alpha \mathfrak M_m + \beta f$. Define
$$
z_1 = \alpha\Omega_{\infty}/f, \, \, z_2 =
\beta\Omega_{\infty}/\mathfrak M_m.
$$
Let $\mathfrak p(z, \mathfrak L)$ denote the Weierstrass $\mathfrak p$-function
of the lattice $\mathfrak L$. Then we have
\begin{equation}\label{m1}
\mathfrak p(z, \mathfrak L) = x - 1/4, \, \, \mathfrak p'(z, \mathfrak L) = 2y + x.
\end{equation}
Write $P_1$ and $P_2$ for the points on $E$ corresponding to $z_1$
and $z_2$. We let $\tau$ be the inverse image of the class $-1 \,
mod \, \mathfrak M_m\CO$ under the isomorphism \eqref{ir}, and we define
$\mathfrak S_m$ to be the fixed field of $\tau$, so that the extension
$\mathfrak R_m/\mathfrak S_m$ has degree 2. Put $\Phi_m = Tr_{\mathfrak
R_m/\mathfrak S_m}(\Psi_m)$. Clearly, we then have
\begin{equation}\label{m2}
\Phi_m = {g_m}^{-1}({\mathcal E}_1^*(z_1 + z_2, \mathfrak L) \, + \,
{\mathcal E}_1^*(z_1 - z_2, \mathfrak L)).
\end{equation}
The next lemma is the heart of our integrality proof.
\begin{lem}\label{m3} We have
\begin{equation}\label{m4}
{\mathcal E}_1^*(z_1 + z_2, \mathfrak L) \, + \, {\mathcal E}_1^*(z_1 -
z_2, \mathfrak L) = 2{\mathcal E}_1^*(z_1, \mathfrak L) + (2y(P_1) +
x(P_1))/(x(P_1) - x (P_2)).
\end{equation}
\end{lem}
\begin{proof} Let $\zeta(z, \mathfrak L)$ denote the Weierstrass zeta function of $\mathfrak L$.  The following identity is classical
$$
{\mathcal E}_1^*(z, \mathfrak L) = \zeta(z, \mathfrak L) - z s_2(\mathfrak L) -
\bar{z}\mathfrak A(\mathfrak L)^{-1},
$$
(see, for example, Prop. 1.5 of \cite{GS}, where the definitions of
the constants $s_2(\mathfrak L)$ and $\mathfrak A(\mathfrak L)$ are also given).
Similarly, we have the addition formula
$$
\zeta(u+v, \mathfrak L) = \zeta(u, \mathfrak L) + \zeta(v, \mathfrak L) +
\frac{1}{2} \frac{\mathfrak p'(u, \mathfrak L) -\mathfrak p'(v,\mathfrak L)}{\mathfrak
p(u, \mathfrak L) -\mathfrak p(v ,\mathfrak L)}.
$$
We apply the first of these formulae when $z = z_1 + z_2$, and $z =
z_1 - z_2$, and then the second of these formulae with $u = z_1, v =
z_2$, and $u = z_1, v = -z_2$, and use the fact that $\zeta(z, \mathfrak
L), {\mathfrak p}'(z, \mathfrak L)$ are both odd functions of $z$, whereas
$\mathfrak p(z, \mathfrak L)$ is an even function of $z$.

We can now complete the proof of Proposition \ref{i}. Since $\mathfrak
S_m$ is Galois over $\mathfrak J_m$, it plainly suffices to show that
$\Phi_m$ is integral at each place of the field $\mathfrak S_m$ above $2$.
In view of \eqref{m2} and \eqref{m4} and the fact that $(g_m, 2)=1$,
it suffices to show that both
$$
C = 2{\mathcal E}_1^*(z_1, \mathfrak L), \, \,  \,  D_m = (2y(P_1) +
x(P_1))/(x(P_1) - x (P_2))
$$
are integral at all places of $\mathfrak S_m$ above $2$. Now, since $z_1$
corresponds to the point $P_1$ of finite odd order on $A$, the
arguments given in section 3 of \cite{CKLZ} show that $C$ is indeed
integral at all places above $2$. Moreover, as $A$ has good reduction
at 2, and $P_1$ and $P_2$ are points on $A$ of finite odd order, we
see that all of the coordinates $x(P_1), x(P_2), y(P_1)$ are
integral at all places above $2$. Suppose there was a place $w$ of
$\mathfrak S_m$ above $2$ where $\ord_w(x(P_1) - x(P_2)) > 0$. Then, under
reduction modulo $w$, we would have that the $x$-coordinates of the
reductions $\bar{P_1}$ and $\bar{P_2}$ of $P_1$ and $P_2$ would be
equal. But, by the explicit group law for the reduced curve, this
means that either $\bar{P_1} = \bar{P_2}$ or $\bar{P_1} + \bar{P_2}
= 0$. This would then imply that either $P_1 - P_2$ or $P_1 + P_2$
are equal to a point of finite order on the formal group of $A$ at
$w$, and since all points of finite order on the formal group $A$ at
$w$ are necessarily of $2$-power order, this is plainly impossible
because $(f, \mathfrak M_m) = 1$. Hence we conclude that $x(P_1) -
x(P_2)$ is a unit at all places of $\mathfrak S_m$ above $2$, and so $D_m$
is integral at all places of $\mathfrak S_m$ above $2$. This completes
the proof.

\end{proof}

For each integer $m \geq 0$, let $\mathfrak D_m$ be the set of all
divisors of $\mathfrak M_m$, which are given by the product over all
elements of any subset of $S_m=\{\pi_1, \ldots, \pi_m\}$. For $M \in
\mathfrak D_m$, we write $L_{S_m}(\bar{\psi}_M, s)$ for the imprimitive
complex $L$-function of the complex conjugate of the
Grossencharacter $\psi_M$, where by imprimitive we mean that the Euler factors of the primes in
the set $S_m$ are omitted from its Euler product. Then, as is shown in
Theorem 2.4 of \cite{CKLZ}, we have the identity
\begin{equation}\label{av}
\sum_{M \in {\mathfrak D}_{m}} L_{S_m}(\bar{\psi}_M,
1)/{\Omega_{\infty}}  = 2^m\Psi_m.
\end{equation}
We now make repeated use of this identity and Proposition \ref{i} to
establish the analytic parts of Theorems \ref{main3} and
\ref{main4}, via a series of induction arguments. The next result is
the key to Theorem \ref{main3}.

\begin{thm}\label{ii} Let $q_1,..., q_r$ be $r \geq 0$ distinct primes, which are $\equiv 1 \mod 4$ and inert in $F$, and put $R = q_1\cdots q_r$. Then
\begin{equation}\label{iii}
ord_2(L^{(alg)}(A^{(R)}, 1)) =  r-1.
\end{equation}
In particular, we have $L(A^{(R)}, 1) \neq 0$.
\end{thm}

\noindent We first observe that this result, when combined with
Corollary \ref{descent3}, implies Theorem \ref{main3}. Indeed,
Theorem \ref{ii} shows that $L(A^{(R)}, 1) \neq 0$,  and so Rubin's
work \cite{R} shows that the $p$-part of the conjecture of Birch and
Swinnerton-Dyer holds for all odd primes $p$. Hence the full
conjecture of Birch and Swinnerton-Dyer will hold for $A^{(R)}$ if and
only if the $2$-part of this conjecture is true. But, as
$\Sha(A^{(R)})(2) =0$ by Corollary \ref{descent3} and the Tamagawa
factors of $A^{(R)}$ at the bad primes are given by $c_7 = 2,
c_{q_i} = 2 \, (1 \leq i \leq r)$, we see that the $2$-part of the
conjecture is just the assertion that $ord_2(L^{(alg)}(A^{(R)}, 1))
= r-1$, as required.

\medskip

Before beginning the proof of Theorem \ref{ii}, we note the
following basic elementary fact. Let $\mathfrak B$ be any elliptic curve
defined over $\BQ$ with complex multiplication by the field $F$, and
let $\phi$ denote the Grossencharacter of $\mathfrak B$. Let $q \geq 5$
be a prime number, where $\mathfrak B$ has good reduction, and which is
inert in $F$. Then we always have
\begin{equation}\label{ss}
\phi((q)) = - q
\end{equation}
where $(q)$ denotes the ideal $q\CO$. Indeed, $\mathfrak B$ has
supersingular reduction at such a prime $q$, and so $q$ must divide
$a_q$, where $a_q$ denotes the trace of Frobenius at $q$ for
$\mathfrak B$. But, by Hasse's theorem $|a_q| \leq 2\sqrt{q}$, and so
$a_q=0$ since $q \geq 5$. Hence the Euler factor at $q$ of the
complex $L$-series of $\mathfrak B$ over $\BQ$ must be equal to $(1 +
q^{1-2s})^{-1}$. But this complex $L$-series must coincide with the
complex Hecke $L$-function $L(\phi, s)$, whose Euler factor at $(q)$
is $(1- \phi((q))q^{-2s})^{-1}$, and so \eqref{ss} follows. Note
that \eqref{ss} immediately implies that, for such a prime $q$, we
have
\begin{equation}\label{ss1}
ord_2(1- \bar{\phi}((q))/q^2) = 1 \, \,  \textrm{whenever} \, \, q
\equiv 1 \ \mod 4.
\end{equation}

\medskip

We now give the proof of Theorem \ref{ii} by induction on $r$, the
assertion being true for $r=0$ because $L^{(alg)}(A, 1) = 1/2$.
Assume next that $r=1$. Applying \eqref{av} with $m=1$ and $\pi_1 =
q_1$, we obtain
\begin{equation}\label{ss2}
L^{(alg)}(A, 1) (1- \bar{\psi}((q_1))/q_1^2)) +
L^{(alg)}(A^{(q_1)},1)/\sqrt{q_1} = 2\Psi_1.
\end{equation}
Writing $v$ for some place of the field $\mathfrak J_1$ above $2$,
Proposition \ref{i} assures us that $\ord_v(\Psi_1) \geq 0$. On the
other hand, since $L^{(alg)}(A, 1) = 1/2$, we conclude from
\eqref{ss1} that
$$
\ord_v(L^{(alg)}(A, 1) (1- \bar{\psi}((q_1))/q_1^2))=0.
$$
Thus \eqref{ss2} immediately implies that $\ord_v(L^{(alg)}(A^{(q_1)},1))=0$, as required. Now assume that $r \geq 2$, and that Theorem \ref{ss} has been proven for all products of $<r$ such primes $q_i$. Applying \eqref{av} with $m=r$ and $\pi_1 = q_1, \ldots, \pi_r = q_r$, we conclude that
\begin{equation}\label{ss3}
L^{(alg)}(A, 1)\prod_{j=1}^{r}(1- \bar{\psi}((q_j))/q_j^2) + W_r +
L^{(alg)}(A^{(R)},1)/\sqrt{R} = 2^r\Psi_r,
\end{equation}
where, writing $U_r$ for the set of all positive divisors of $R$
distinct from $1$ and $R$, we have
\begin{equation}\label{ss4}
W_r = \sum_{M \in U_r}L_{S_r}(\bar{\psi}_M, 1)/\Omega_\infty.
\end{equation}
Now again it follows from \eqref{ss1} that
$$
ord_2(L^{(alg)}(A, 1)\prod_{j=1}^{r}(1- \bar{\psi}((q_j))/q_j^2)) =
r-1.
$$
Moreover, if $v$ denotes any place of $\mathfrak J_r$ above $2$,
Proposition \ref{i} tells us that $\ord_v(\Psi_r) \geq 0$. Hence
\eqref{ss4} will immediately imply the desired result \eqref{iii},
once we have established the following lemma.

\begin{lem}
For all $r \geq 2$, and all places $v$ of $\mathfrak J_r$ above $2$, we have $\ord_v(W_r) \geq r$.
\end{lem}
\begin{proof} For $M \in U_r$, put $\Lambda_M = L_{S_r}(\bar{\psi}_M, 1)/\Omega_\infty$, so that
$$
\Lambda_M = L(\bar{\psi}_M, 1)/\Omega_\infty \times
\prod_{q|\frac{R}{M}}(1- \bar{\psi}_M((q))/q^2).
$$
Now $M$ has strictly less than $r$ prime factors, and hence we
conclude from our inductive hypothesis and \eqref{ss1} that
$\ord_v(\Lambda_M) = r-1$. But, since $\psi_M$ is the
Grossencharacter of an elliptic curve defined over $\BQ$, we know
that $\sqrt{M}\Lambda_M \in \BQ$. It follows easily from these last
two assertions that we can write
\begin{equation}\label{ss5}
\Lambda_M = 2^{r-1}\sqrt{M} + 2^r\alpha_M,
\end{equation}
where $\alpha_M$ is some element of $\mathfrak J_r$ with
$\ord_v(\alpha_M) \geq 0$. Thus, in order to show that $W_r = \sum_{M
\in U_r}\Lambda_M$ satisfies $\ord_v(W_r) \geq r$, it suffices to
prove that
\begin{equation}\label{ss6}
\ord_v(\sum_{M \in U_r}\sqrt{M}) \geq 1.
\end{equation}
But clearly
\begin{equation}\label{ss7}
(\sum_{M \in U_r}\sqrt{M})^2 = \sum_{M \in U_r}M  + 2\gamma_r,
\end{equation}
where $\ord_v(\gamma_r) \geq 0$, and
\begin{equation}\label{ss8}
\sum_{M \in U_r}\Lambda_M \equiv (\sum_{M \in U_r} 1) \ \mod 2 \equiv 0 \ \mod 2,
\end{equation}
because $U_r$ has cardinality equal to $2^r - 2$. Recalling that $2$
is unramified in $\mathfrak J_r$, the inequality \eqref{ss6} now follows
immediately from \eqref{ss7} and \eqref{ss8}. This completes the
proof of the lemma, and so also the proof of Theorem \ref{ii}.
\end{proof}

 \medskip

We now turn to the proof of the analytic part of Theorem
\ref{main4}. As in Theorem \ref{main3}, let $q_1, \ldots, q_r$ be $r
\geq 0$ distinct prime numbers which are $\equiv 1 \mod 4$ and
which are inert in $F$, and put $R=q_1\cdots q_r$. Let
$p_1, \ldots, p_k$ be $k \geq 1$ prime numbers, which split
completely in the field
$$
\mathfrak H = \BQ(A[4], \sqrt{q_1}, \ldots \sqrt{q_r}).
$$
In particular, these latter primes split in $F$, and we write $\mathfrak
p_1, \ldots, \mathfrak p_{2k}$ for the set of primes of $F$ lying above
them in some order. Put $\rho_j = \psi(\mathfrak p_j)$, and note that
$\rho _j \equiv 1 \mod 4$ by the theory of complex
multiplication, since $\mathfrak p_j$ splits completely in $\mathfrak H$.
For each integer $n$ with $1 \leq n \leq 2k$, define $\mathfrak N_n =
\rho_1\cdots\rho_n$. We now prove by induction on both $r\geq 0$ and
$n \geq 1$ that
\begin{equation}\label{ix}
\ord_v(L(\bar{\psi}_{R \mathfrak N_n}, 1)/\Omega_\infty) \geq r+n.
\end{equation}
Taking $n =2k$, we immediately obtain the statement of Theorem
\ref{main4} as a special case of \eqref{ix}. We consider all
divisors $M$ of $R\mathfrak N_n$ which are given by the product of all
the elements of an arbitrary subset of $S_{r,n}= \{q_1,..., q_r,
\rho_1,..., \rho_n\}.$ We write $L_{S_{r,n}}(\bar{\psi}_M, s)$ for
the $L$-function of the complex conjugate of the Grossencharacter
$\psi_M$, but with the Euler factors for the primes in the set
$S_{r,n}$ omitted from its Euler product. Let $\mathfrak J_{r,n} =
F(\sqrt{q_1}, \ldots, \sqrt{q_r}, \sqrt{\rho_1}, \ldots,
\sqrt{\rho_n})$. Then the equation \eqref{av} gives in the present
situation
\begin{equation}\label{x}
\sum_{M} L_{S_{r,n}}(\bar{\psi}_M, 1)/{\Omega_{\infty}}  =
2^{r+n}\Psi_{r,n},
\end{equation}
where $\Psi_{r,n}$ is the trace from $F(A_{g_{r,n}})$ to $\mathfrak
J_{r,n}$ of $\mathcal {E}_1^*(\Omega_\infty/g_{r,n}, \mathfrak L)$, with
$g_{r,n} = R\mathfrak N_n\sqrt{-7}$. By Proposition \ref{i} and the fact
that $n \geq 1$, we have $\ord_v(\Psi_{r,n}) \geq 0$ for all places
$v$ of $\mathfrak J_{r,n}$ above $2$. In order to analyse the order at
$v$ of the terms on the left hand side of \eqref{x}, we make the
following observation on Euler factors. Suppose first that $M$
divides $R$. Then $\psi_M(\mathfrak p_j) \equiv 1 \mod 4$ for $1
\leq j \leq n$ because $\mathfrak p_j$ splits completely in the field
$F(E^{(M)}[4])$, which is clearly a subfield of $\mathfrak H$. Hence the
Euler factor
\begin{equation}\label{xi}
1- \bar{\psi}_M(\mathfrak p_j)/p_j = (\psi_M(\mathfrak p_j) -
1)/\psi_M(\mathfrak p_j)
\end{equation}
is always divisible by $4$. It is easy to see that all other Euler
factors which occur in the imprimitive $L$-functions on the left
hand side of \eqref{x} will be divisible at least by 2. Consider now
any one of the $M \neq R\mathfrak N_{n}$ occurring in the sum on the
left hand side of \eqref{x}, and let $r(M)$ be the number of factors
of $M$ lying in the set $\{q_1,..., q_r\}$, and let $n(M)$ be the
number of its factors lying in the set $\{\rho_1,..., \rho_n\}$. If
$n(M) = 0$, then Theorem \ref{ii}, and the fact the Euler factors
given in $\eqref{xi}$ are divisible by 4, imply that
\begin{equation}\label{x2}
\ord_v(L_{S_{r,n}}(\bar{\psi}_M, 1)/{\Omega_{\infty}}) \geq r(M)-1 +
2n  + (r-r(M)) \geq n+r.
\end{equation}
On the other hand, if $n(M) > 0$, we have
\begin{equation}\label{x3}
\ord_v(L_{S_{r,n}}(\bar{\psi}_M, 1)/{\Omega_{\infty}}) \geq
\ord_v(L(\bar{\psi}_M, 1)/\Omega_\infty) + (n-n(M)) + (r-r(M)).
\end{equation}
One first proves \eqref{ix} for $r=0$ and all $n \geq 1$ by
induction on $n$. To do this we use \eqref{x} with $r=0$, noting
that the induction starts, because when $n=1$ there are just two
terms on the left hand side of \eqref{x}, and the term with $M=1$ is
handled by \eqref{x2}. When $n > 1$, one has to use \eqref{x2} for
the term with $M=1$, and \eqref{x3} plus the inductive hypothesis to
handle the terms with $M \neq 1, \mathfrak R_n$. We now assume that $r >
0$, and proceeds to prove by induction on $n$ that \eqref{ix} holds
for all $n \geq 1$, using now \eqref{x} for the given $r$. One sees
easily again that the induction starts with \eqref{ix} being valid
for $n=1$. When $n > 1$, one uses \eqref{x2} to handle the terms on
the left hand side of \eqref{x} for which $M$ divides $R$. For the
terms on the left hand side of \eqref{x} with $M$ not dividing $R$,
but $M \neq R\mathfrak R_n$, we can use \eqref{x3} together with the
inductive hypothesis that
$$
\ord_v(L(\bar{\psi}_M, 1)/\Omega_\infty) \geq r(M) + n(M).
$$
In this way, one sees that all terms on the left hand side of
\eqref{x}, except the term with $M = R{\mathfrak R}_n$, have
$v$-order at least $r+n$. But the right hand side of \eqref{x}  also
has $v$-order at least $r+n$ because $\ord_v(\Psi_{r,n}) \geq 0$.
Hence the remaining term on the left had side of \eqref{x} must also
have $v$-order at least $r+n$, completing the inductive proof of
\eqref{ix}, and so also the proof of  the lower bound of Theorem
\ref{main4}. If we assume in addition that $L(A^{(M)}, 1) \neq 0$.
Then, by Kolyvagin's theorem, $A^{(M)}(\BQ)$ is finite, and so we
conclude from Corollary \ref{descent6} that $\Sha(A^{(M)})(2) \neq
0$ because $N_+ > 1$, completing the proof of Theorem \ref{main4}.

\medskip

We end this section by pointing out that, if we assume that
$L(A^{(M)}, 1) \neq 0$, the unproven $2$-part of the conjecture of
Birch and Swinnerton-Dyer has an interesting consequence for the
quadratic twists of the curve $A = X_0(49)$, which appear in Theorem
\ref{main4}. Indeed, by Corollary \ref{descent6}, we have
$\Sha(A^{(M)})(2) \neq 0$ for these twists.  On the other hand,
using Proposition \ref{descent8}, we see that the $2$-part of the
conjecture of Birch and Swinnerton-Dyer predicts
that
\begin{equation}\label{2sd}
ord_2(L^{(alg)}(A^{(M)}, 1)) = 2k + r -1 +
ord_2(\#(\Sha(A^{(M)})(2))).
\end{equation}
In addition, the Cassels-Tate pairing tells us that
$ord_2(\#(\Sha(A^{(M)})(2)))$ must be an even integer. Thus, in view
of Theorem \ref{main4}, we see that not only does this conjecture
predict that $\Sha(A^{(M)})(2) \neq 0$ when $k \geq 1$, as we have
proven, but it also predicts the stronger lower bound
$$
ord_2(L^{(alg)}(A^{(M)}, 1)) \geq r + 2k + 1.
$$
We remark that it does not seem easy to prove this sharper lower
bound by Zhao's method. However, we shall see in the next section that one can derive this sharper lower bound by using Waldspurger's formula.

\section{Method using Waldspurger's formula}

In this section, we shall first establish an explicit Waldspurger formula for the family of quadratic twists of the curve $A=(X_{0}(49),[\infty])$, and then use it to prove some results on $L$-values in this family, including all of those established in the last section by Zhao's method. We begin by using Gross-Prasad theory to get an appropriate test vector for our formula.

\medskip

If $W$ is any abelian group, $\hat{W}$ will denote the tensor product over $\BZ$ of $W$ with $\hat{\BZ} = \prod_p \BZ_p$. Now consider a definite quaternion algebra $B$ over $\BQ$, and an open subgroup $U$ of $\wh{B}^\times$. Let $X$ denote the finite set $B^\times \bs \wh{B}^\times/U$, and write $g_1, \ldots, g_n$ for a set of representatives of $X$ and $[g_1], \ldots, [g_n]$ for their classes in $X$. Denote by $\BZ[X]$ the free $\BZ$-module of formal sums $\sum_{i=1}^n a_i [g_i]$ with $a_i\in \BZ$, and $\BZ[X]^0$ its degree 0 sub-module (here the degree of $\sum a_i[g]_i$ is $\sum a_i$). Define $w_i:=\# ((B^\times \cap g_iU g_i^{-1})/\pm 1)$ and let $\pair{\ , \ }$ be the $\BZ$-bilinear pairing on $\BZ[X]$ defined by $\pair{[g_i], [g_j]}=\delta_{ij} w_i$. Let $\pi$ be an automorphic representation of $B_\BA^\times$ whose Jacquet-Langlands correspondence for $\GL_2(\BA)$ is associated with an elliptic curve over $\BQ$. There is a natural embedding
$$\pi^U \ra \BC[X]^0=\BZ[X]^0\otimes_\BZ \BC, \qquad f \mapsto \sum f([g_i]) w_i^{-1}[g_i].$$

Now take  $B$ to be the quaternion algebra over $\BQ$ ramified exactly at
$\infty$ and $7$, i.e.
$$
B=\BQ+\BQ i +\BQ j+\BQ k, \ \ \ i^2=-1, \ j^2=-7, \ ij=-ji=k.
$$
Let $\pi = \otimes_v\pi_v$ be the automorphic representation of $B^\times_\BA$ corresponding to $A=(X_0(49),[\infty])$ via modularity of $A$ and the Jacquect-Langlands correspondence.   This automorphic representation is naturally realized as a subspace of  the space of the infinitely differentiable complex-valued functions $C^{\infty}(B^\times\backslash \widehat{B}^{\times}/\widehat{\BQ}^{\times})$. Let $\CO_B$ denote the maximal order $\BZ[1, i, (i+j)/2, (1+k)/2]$
of $B$, and note that $\CO_B^\times=\{\pm 1, \pm i\}\cong \mu_4$. The local representations $\pi_v$ have the following properties:
\begin{enumerate}
  \item $\pi_\infty$ is trivial;
  \item $\pi_p$ is spherical if $p\neq\infty,7$, i.e. $\pi^{\CO^{\times}_{B_p}}$ is dimension one;
  \item $\pi_7$ has conductor with exponent $1$, i.e. for a uniformizer $j$ at $7$,
  $\pi_{7}^{1+j\CO_{B_7}} \neq0$ but $\pi_{7}^{\CO^\times_{B_7}}=0$.
\end{enumerate}
Let $U=\prod_p U_p$ be the open compact subgroup of $\wh{B}^\times$
defined by $U_p = \CO_{B_p}^\times$ if $p\neq 7$, and $U_p = 1+j \CO_{B_7}$ if $p=7$.
Then $\pi^U\neq 0$ is a representation of $B^{\times}_7.$  The next theorem gives a description of the space $\pi^U$, which will be important for computing the Gross-Prasad test vector.  Let $\chi_0$ be the character of $B_\BA^\times$ attached to the
quadratic extension $\BQ(\sqrt{-7})$, i.e. the composition of the
following morphisms:
\begin{equation}\label{tc}
\wh{B}^\times \stackrel{\det}{\lra}
\wh{\BQ}^\times=\BQ^\times_+\times \wh{\BZ}^\times \lra
\wh{\BZ}^\times \lra (\BZ/7\BZ)^\times \lra
\BF_7^\times/\BF_7^{\times 2} \stackrel{\sim}{\lra} \pm 1.
\end{equation}
Since $A$ has CM by $\BZ[\sqrt{-7}]$, we have $\pi\cong \pi\otimes \chi_0$. 
\begin{thm}\label{automorphic}The vector space $\pi^U$ is a two-dimensional irreducible
representation of $B_7^\times$, which has orthogonal basis $f_0,
f_1$ defined as follows.
There is a natural bijection
$$
\Lambda:=(\CO^\times_B/\pm 1)\bs (\CO_{B_7}^\times/\BZ_7^\times U_7) \lra B^\times\bs
\wh{B}^\times/\wh{\BQ}^\times U
$$
induced by the embedding
$\CO_{B_7}^\times \subset B_7^\times\ra \wh{B}^\times$, and
$\CO_{B_7}^\times/\BZ_7^\times U_7$ is a cyclic group of order $8$
so that $\Lambda\cong \BZ/4\BZ$. Via the above bijection, the form
$f_0$ (resp. $f_1$) is supported on the elements of $\Lambda$ of
order dividing $2$ (resp. of exact order $4$), valued in $0, \pm 1$, and satisfies $\sum_{\lambda\in \Lambda} f_i(\lambda)=0$  $(i=0,1)$. Such $f_i$'s
are unique up to multiplication by $\pm 1$.  Moreover,  $\chi_{0}f_{0}=f_{0}$ and $\chi_{0}f_{1}=-f_{1}.$

\end{thm}

\begin{proof}
Note that the class number of $\CO_B$ is one (see,
\cite{Vigneras}) and therefore $\wh{B}^\times=B^\times
\wh{\CO}_B^\times$. Let $U^{(7)} = \prod_{v \neq 7}U_v$.  It is easy to see that the embedding
$B_7^\times\ra \wh{B}^\times$ induces a bijective map:
$$
H\bs B_7^\times/\BQ_7^\times U_7 \lra B^\times \bs \wh{B}^\times/\wh{\BQ}^\times U,
$$
where $H=B^\times \cap B_7^\times
U^{(7)}\subset B_7^\times$ is a semi-product of $\CO_B^\times$ with
$j^\BZ=B_7^\times/\CO_{B_7}^\times$. Let $\mathfrak O$ denote the ring
$\BZ_7[i]\subset \CO_{B_7}$, then $\CO_{B_7}=\mathfrak O+j\mathfrak O$ and one can
see that $\CO_{B_7}^\times/\BZ_7^\times U_7 \cong
\mathfrak O^\times/\BZ_7^\times\cong \BF_{7^2}^\times/\BF_7^\times$ is a
cyclic group of order $8$.
Note also that via the above identification, for any two automorphic forms $f, f'$,  the pairing $ \pair{f, f'}=\sum_{\lambda\in\Lambda} f(\lambda){f'(\lambda)}$, since the $w_i$'s in the definition of $\pair{\ ,\ }$ are all equal to one. Moreover, forms
$\chi\circ \det$ become $\chi\circ
\RN_{\BF^\times_{7^2}/\BF_7^\times}$. Note also that $\pi^U$ is
orthogonal to all $\chi\circ \RN$, and $\pi^U$ has dimension $2$ by Theorem 3.6 in \cite{Tunnell}, and it is easy to see that the forms $f_0, f_1$ give an orthogonal basis of $\pi^U$.
It is clear that $\chi_0$ is invariant under the left action of
$B^\times \cdot H$,  and under the right action of $U$. Observe further that
$\BZ_7[i]^\times \lra \CO_{B_7}^\times/\BZ_7^\times U_7$ is
surjective. Thus the induced map of $\chi_0$ on $\BZ_7[i]^\times$ factors
through $\BZ_7[i]^\times\lra \BF_{7^2}^\times \stackrel{\RN}{\lra}
\BF_7^\times/\BF_7^{\times 2}$, whence $\chi_0$ takes value $1$ on the
support of $f_0$ and $-1$ on the support of $f_1$.

\end{proof}

Now Let $K$ be an imaginary quadratic field in which $7$ is ramified, and let $K_7$ be its completion at the unique prime above $7$. It is easily seen that $K_7$ is isomorphic to either
$\BQ_7(\sqrt{-7})$ or $\BQ_7(\sqrt{-35})$.  We denote by $-D$ the discriminant of $K$. Let $d$ be any positive fundamental discriminant dividing $D$, and let
$\chi$ be the quadratic character of $K$ corresponding to the unramified extension $K(\sqrt{d})/K$. We write $L(A/K, \chi, s)$ for the complex $L$-function of $A/K$ twisted by $\chi$. Thus we have, by the induction property of $L$-series,
$$
L(A/K, \chi, s) = L(A^{(d)},s)L(A^{(-D/d)}, s).
$$
Since $7$ divides $D$, and $A^{(m)}$ is isogenous to $A^{(-7m)}$ for any integer $m$ prime to $7$, we conclude that $L(A/K, \chi, s)$ has global root number $+1$.

\medskip

Fix an embedding of $K$ into $B$  such that $\CO_K$ is embedded into $\CO_B$.  Via this embedding we view $K^\times$ as a sub-torus of $B^\times$ and $K^{\times}_7$ as a sub-torus of $B_{7}^{\times}$.  Recall $\chi$ as before, let $\chi_7$ be the $7$-component of $\chi$,  define  $\pi^{U, \chi_{7}}$ to be the  vector space
$$\{f\in\pi^{U} | \pi(t)f=\chi_{7}(t)f, \ \forall t\in K^{\times}_{7} \}.$$
It is then known, by Gross-Prasad theory of test vectors,  that $\pi^{U, \chi_7}$ has dimension $1$.

\begin{defn}
Any nonzero vector in the vector space $\pi^{U, \chi_{7}}$ is called a Gross-Prasad test vector for $(\pi,\chi)$ or $(A, \chi)$. A Gross-Prasad test vector for $(A,\chi)$ is called a primitive Gross-Prasad test vector  if its values are integers which generate $\BZ.$
\end{defn}

\noindent We note that a primitive Gross-Prasad test vector is unique up to multiplication by $\pm1$.

\begin{thm}\label{test vector}
Let $K$ be an imaginary quadratic field, in which $7$ is ramified,
and let $\chi$  be an unramified quadratic character of $K$ such that $L(A, \chi, s)$ has global root number equal to $+1$. Let $f$ be a primitive test vector for $(A, \chi)$, then
$\pair{f, f}= 2$ if $K_7=\BQ_7(\sqrt{-7})$, and $\pair{f, f}= 4$ if $K_7=\BQ_7(\sqrt{-35})$.
\end{thm}

\begin{proof}
Let $\iota'_7$ be the local embedding $K_7\ra B_7$, which is given respectively by$\sqrt{-7}\mapsto j$, or $\sqrt{-35}\mapsto j+2k$.
Fix an embedding $\iota: K\ra B$ such that $\iota(\CO_K)\subset \CO_B$, where $\CO_K$ denotes the ring of integers of $K$. Then the local component $\iota_7$ of $\iota$ at $7$ is conjugate to $\iota'_7$, say by an element $g\in B_7^\times$, $\iota_7'=g^{-1}\iota_7 g$. By Theorem \ref{automorphic}, it is easy to check that $f: x\mapsto f'(xg)$ is a primitive test vector for $(A, \chi)$, where
$$f'=\begin{cases} f_0,  \qquad &\text{if $K_7\cong\BQ_7(\sqrt{-7})$,  $\chi_7(\varpi)=1$}, \\
f_1, \qquad  &\text{if $K_7\cong\BQ_7(\sqrt{-7})$, $\chi_7(\varpi)=-1$},\\
f_0-f_1, \qquad &\text{if $K_7\cong\BQ_7(\sqrt{-35})$, $\chi_7(\varpi)=1$},\\
f_0+f_1, \qquad &\text{if $K_7\cong\BQ_7(\sqrt{-35})$, $\chi_7(\varpi)=-1$},
\end{cases}$$ and  $\varpi$ is a uniformizer of $K_7$.
The assertion about  $\pair{f, f}=\pair{f', f'}$ now follows from Theorem \ref{automorphic}.
\end{proof}

\begin{thm}[Explicit Waldspurger Formula] \label{waldspurger}
Let $A=(X_{0}(49),[\infty])$,  and let $K$  an imaginary
quadratic field with discriminant $-D$, in which $7$ is ramified. Let $d$ be any positive fundamental discriminant dividing $D$, and let
$\chi$ be the quadratic character of $K$ corresponding to the unramified extension $K(\sqrt{d})/K$.
Let $f$ be the Gross-Prasad test vector for $(A,\chi)$
as in Theorem \ref{test vector}. Then we have
$$
\left|\sum_{t\in \widehat{K}^{\times}/K^{\times}\widehat{\CO}^{\times}_{B}}f(t)\chi(t) \right|^{2}=2^{2+\delta}L^{(alg)}(A^{(d)}, 1)L^{(alg)}(A^{(-D/d)}, 1),
$$
where $\delta=0$ if $K_{7}\cong\BQ_{7}(\sqrt{-7})$, and $\delta=1$ if $K_{7}\cong\BQ_{7}(\sqrt{-35})$.
\end{thm}
\begin{proof}
We have the following $L$-value formula, which can be derived from the Waldspurger formula Theorem 1.4 \cite{YZZ} (for details, see Theorem 1.10 in \cite{CST}). If $f$ is the Gross-Prasad test vector for $(A, \chi)$ in Theorem \ref{test vector}, then we have
$$L(A/K,\chi, 1)= \frac{(8\pi^2)(\phi,\phi)_{\Gamma_0(49)}}{2\sqrt{D}}
\cdot\frac{|\sum_{t\in \widehat{K}^\times/K^\times\widehat {\CO}_K^\times}f(t)\chi(t)|^2}{\pair{f,f}},$$
where $\phi$ is the newform associated to $A$,  $(\phi, \phi)_{\Gamma_0(49)}$ is the Petersson inner norm defined by
$$(\phi,\phi)_{\Gamma_0(49)}=\iint_{\Gamma_0(49) \bs \CH}|\phi(x+iy)|^2 dxdy,$$
and $\pair{f, f}$ is the pairing defined at the beginning of this section.

\medskip

Let $A(\BC)^{\pm}$ denote the $\pm$ eigen-subgroups of $A(\BC)$ under the action of the complex conjugation. Thus $A(\BC)^+=A(\BR)$ and $A(\BC)^{-}=A'(\BR)$, where $A'$ is the twist of $A$ by $\BC/\BR$. We denote by $\beta$ the \Neron differential for $A$, and define
$$
\Omega:=\int_{A(\BC)^+}\beta\quad \textrm{and} \quad \Omega^-=\int_{A(\BC)^-}\beta.
$$
In fact, $\Omega = \Omega_\infty$ in our earlier notation since $A(\BR)$ has only one connected component. For the quadratic twist $A^{(d)}$, we can define $\Omega^{(d)}$ and $\Omega^{-(d)}$ analogously. Since $A(\BR)$ and $\sqrt{-7}A(\BR)$ generate an index $2$ subgroup of $H_1(A(\BC), \BZ)$, we have
$$
-i\Omega\Omega^-
=\iint_{A(\BC)}|\beta\wedge\ov{\beta}|=8\pi^2(\phi,\phi)_{\Gamma_0(49)}.$$
Noting that $c_4(A)=105$, $c_6(A)=3^3 7^2$, and using the Main Result 1.1 in \cite{Pal}, we obtain
$\Omega^{(d)}\Omega^{(-D/d)}=\Omega\Omega^-/\sqrt{D}$. Thus
$$\Omega^{(d)}\Omega^{(-D/d)}=  8\pi^2(\phi, \phi)_{\Gamma_0(49)}/i\sqrt{D},$$
and the desired formula follows, on noting Theorem \ref{test vector}.
\end{proof}

Assume now that $n$ is a positive integer with $n\equiv 1\mod 4$, and $n$ a quadratic residue modulo $7$. Let $\iota$ be an embedding  from $K=\BQ(\sqrt{-7n})$ into $B$ such that $\iota(\CO_K)\subset \CO_B$, i.e. the trace zero element $\xi=\iota(\sqrt{-7n})$ belongs to $\CO_B$. Let $p$ be an odd prime such that both $n$ and $-7$ are quadratic non-residues modulo $p$, whence $-7n$ is a quadratic residue modulo $p$.
Let $m\in \BZ$ be such that $m^2+7n\equiv 0\mod p$, so that $m+\xi$ has reduced norm divisible by $p$. Now the class number of left ideals of $\CO_B$ is equal to $1$ (see the table on p.153 of \cite{Vigneras}). Hence there is an element $t\in \CO_B$ of reduced norm $p$, and an element $u\in \CO_B$ such that
$$m+\xi=ut.$$
It follows that $t\xi t^{-1}=tu-m\in \CO_B$, and thus $t\iota t^{-1}: \sqrt{-7n}\mapsto t\xi t^{-1}\in \CO_B$ is another embedding. Suppose that $g\in B_7^\times$ is such that $g^{-1}\iota_7g=(t_7g)^{-1} (t_7\iota_7 t_7^{-1}) (t_7g)$ is the local embedding $\iota_7'$ in Lemma \ref{test vector}. Recall that $\chi_0$ is defined by \eqref{tc}. Note that $\chi_0(t_7)=-1$, whence it follows that one of $\chi_0(g), \chi_0(t_7g)$ must be $+1$. Replacing $\iota$ by $t\iota t^{-1}$ if $\chi_0(g)=-1$, we may assume that $\chi_0(g)=+1$. It follows that, with respect to the embedding $\iota$, and when $K_7\cong \BQ_7(\sqrt{-7})$ and $\chi_7(\varpi)=1$, a primitive test vector $f$ exists satisfying $\chi_0f=f$.

\medskip

\begin{defn}\label{quadratic point} Let $K=\BQ(\sqrt{-7n})$, where  $n$ is a positive integer with $n \equiv 1 \mod 4$, and $n$ a quadratic residue modulo $7$. For each positive divisor $d$ of $n$, let $\chi^{(d)}$
denote the character of $\CA$ defining the quadratic extension $K(\sqrt{d^*})$ over $K$, where $d^*=(-1)^{\frac{d-1}{2}}d$.
Fix an embedding $\iota$ of $K$ into $B$ such that $\iota(\CO_K)\subset \CO_B$, and if $g\in B_7^\times$ is such that $g^{-1}\iota_7 g$ is equal to $\iota_7'$ in the proof of Theorem \ref{test vector}, then also $\chi_0(g)=1$. We then take a primitive Gross-Prasad test vector $f$ for $(A,\chi^{(d)})$ with respect to this given embedding $\iota$.

\end{defn}

\noindent For each positive divisor $d$ of $n$, define

$$y_d:=\sum_{t\in \CA} f(t)\chi^{(d)}(t).$$

\begin{lem}\label{relation for yd}
When $n$ is a quadratic residue modulo $7$, with $\chi_7|_{K^{\times}_7}=1$, we have
$$y_d=y_{n/d}.$$
\end{lem}
\begin{proof}
First we observe that
$$
\chi^{(n/d)}\chi^{(d)}(t)=\frac{\sigma_{t}(\sqrt{d^*})}{\sqrt{d^*}}\frac{\sigma_{t}(\sqrt{(n/d)^*})}{\sqrt{(n/d)^*}}
=\frac{\sigma_{t}(\sqrt{-7})}{\sqrt{-7}}=\chi_{0}(t).
$$
From the argument before Definition \ref{quadratic point} and Theorem \ref{test vector}, the global embedding $\iota$ satisfying the local condition at the place $7$ will give the test vector $f(t)=f_0(tg)$ with $g\in B_7^\times$ and
$\chi_0(g)=1$. Thus we have
$$y_d=\sum_{t\in \CA} f_0(tg)\chi^{(d)}(t)=\sum_{t\in\CA}f_0(tg)\chi^{(n/d)}(t)\chi_0(tg)=\sum_{t\in\CA}f_0(tg)\chi^{(n/d)}(t)=y_{n/d}.$$
\end{proof}
\noindent As a first application of these methods, we give a second proof of Theorem \ref{ii}.
\medskip

\begin{thm}\label{ord2LR+}
Let $R=q_{1}\cdots q_{r} \ (r\geq1)$, be a product of distinct
primes which are congruent to $1$ modulo $4$ and inert in $\BQ(\sqrt{-7})$.
Then we have $ord_2(L^{(alg)}(A^{(R)}, 1)) = r-1$.
\end{thm}
\begin{proof}
Take $K=\BQ(\sqrt{-7R})$. Thus any unramified quadratic extension of $K$
is of the form $K(\sqrt{d})$ for a unique positive
divisor $d$ of $R$. Write $\chi^{(d)}$ for the corresponding character
of $K$, and let $\varpi=\sqrt{-7R}$, so that $\varpi$ is a uniformizer of $K_{7}$.
We will prove the theorem by induction on $r$, beginning with the cases
$r=1$ and $r=2$.

\medskip

First assume that $r=1$. Then $K_{7}\cong\BQ_{7}(\sqrt{-35})$, and one sees easily that
$\chi^{(q_1)}_{7}(\varpi)=-1$. We choose the primitive test vector $f$ for $(A,\chi^{(q_1)})$ according to Theorem \ref{test vector}.
It follows that
$$
y_{1}+y_{q_1}=2\sum_{t\in 2\CA}f(t).
$$
Writing $\chi^{(1)}$ for the trivial character of $K$, we have $\chi^{(1)}(\varpi)=1$, so $y_{1}=0$. Thus, since $h_{4}(-7q)=0$,
we conclude that $ord_2(y_{q_1})=1$. By Theorem \ref{waldspurger}, we have
$$
|y_{q_1}|^{2}=8\cdot L^{(alg)}(A^{(q_1)},1)L^{(alg)}(A,1),
$$
whence, since $L^{(alg)}(A,1) = 1/2$, we conclude that $ord_2(L^{(alg)}(A^{(q_1)},1))=0$, proving the case for $r=1$.

\medskip

Next consider the case $r=2$. Then $K_{7}\cong\BQ_{7}(\sqrt{-7})$, and $\chi^{(q_1q_2)}(\varpi)=1$.
We choose the primitive test vector $f$ for $(A,\chi^{(q_{1}q_{2})})$ according to Theorem \ref{test vector}.
First note that $h_{4}(-7q_{1}q_{2})=0$, so we have
$$
\sum_{d\mid q_{1}q_{2}, \tau(d) \  \textrm{even}}y_{d}=\sum_{d\mid q_{1}q_{2}}y_{d}=2^{2}\sum_{t\in 2\CA}f(t),
$$
where $\tau(d)$ denotes the number of odd prime factors of $d$. Observing that $f$ is supported on $2\CA$, and that $y_{d}=0$ when $\tau(d)$ odd, we conclude from the assertion for $r=1$ that
$$
y_{1}+y_{q_{1}q_{2}}=2^{2}\sum_{t\in 2\CA}f(t).
$$
Note also from Lemma \ref{relation for yd} implies that $y_{q_{1}q_{2}}=y_{1}$.
It then follows that $ord_2(y_{q_{1}q_{2}})=1$, and thus, by the equation
$$
|y_{q_{1}q_{2}}|^{2}=4\cdot L^{(alg)}(A^{(q_{1}q_{2})},1)L^{(alg)}(A,1),
$$
we conclude that $ord_2(L^{(alg)}(A^{(q_{1}q_{2})},1))=1$, establishing the case $r=2$.

\medskip

Now assume $r>2$, and assume by induction that the theorem is valid for all positive divisors $d$ of $R$ with $d \neq R$. Suppose first that $\tau(R)$ is even, so that $K_{7}\cong\BQ_{7}(\sqrt{-7}), \chi^{(R)}(\varpi)=1$. We then take the primitive Gross-Prasad test vector $f$ for $(A,\chi^{(R)})$ according to Theorem \ref{test vector}. It follows that
$$
\sum_{d\mid R, \  \tau(d) \  \textrm{even}}y_{d}=2^{r}\sum_{t\in 2\CA}f(t).
$$
Note that, from Lemma \ref{relation for yd}, $y_{d}=y_{R/d}$, when $\tau(d)$ is even, so that the
above formula becomes
\begin{equation}\label{x4}
\sum_{\sqrt{R}\leq d\mid R, \  \tau(d) \  \textrm{even} }y_{d}=2^{r-1}\sum_{t\in 2\CA}f(t).
\end{equation}
By our induction hypothesis, and using the formula
$$
|y_{d}|^{2}=4L^{(alg)}(A^{(d)},1)L^{(alg)}(A^{(R/d)},1),
$$
we conclude that $ord_2(y_{d})=r/2< r-1$ for all positive even divisors $d$ of $R$ with $d \neq R$. Thus \eqref{x4} gives
$$
y_{R}\equiv -\sum_{d\neq R}y_{d} \ \mod 2^{r-1}.
$$
Notice that we omit the same condition under the sum with \eqref{x4},
since we know that each term of the right hand side of the above equation
has valuation exactly $r/2$, and the number of all terms is odd.
Hence $ord_2(y_{R})=r/2$, and then Theorem \ref{waldspurger} gives
$$
ord_2(L^{(alg)}(A^{(R)},1))=r-1.
$$
Suppose next that $\tau(R)$ is odd, so that we have
$K_{7}\cong\BQ_{7}(\sqrt{-35})$, and $\chi^{(R)}(\varpi)=-1$. We choose the primitive Gross-Prasad test
vector $f$ for $(A,\chi^{(R)})$ according to Theorem \ref{test vector}. Hence
$$
\sum_{d\mid R, \  \tau(d) \  \textrm{odd}}y_{d}=2^{r}\sum_{t\in 2\CA}f(t).
$$
By our induction hypothesis, for positive $d$ dividing $R$ with $d\neq R$,
and the formula
$$
|y_{d}|^{2}=8L^{(alg)}(A^{(d)},1)L^{(alg)}(A^{(R/d)},1),
$$
we have $ord_2(y_{d})=(r+1)/2< r$. Hence
$$
y_{R}\equiv -\sum_{d\neq R}y_{d} \ \mod 2^{r}.
$$
Note that we use the same conventions as in the even case,
and we obtain $ord_2(y_{R})=(r+1)/2$. Finally, applying Waldspurger's formula, we conclude that
$$
ord_2(L^{(alg)}(A^{(R)},1))=r-1.
$$
This completes the proof.
\end{proof}

The following proposition is proved in \cite{CKLZ} and we now give
a different proof by applying Theorem \ref{waldspurger}.

\begin{prop}\label{w}
Let $N=p_1\cdots p_k$ be a product of
distinct primes which are all  $\equiv 1\mod 4$, and
quadratic residues modulo $7$. Then
$ord_2(L^{(alg)}(A^{(n)}, 1) \geq 2k-1$.
\end{prop}
\begin{proof}
Let $K=\BQ(\sqrt{-7N})$. Again any unramified quadratic
extension of $K$ is of the form $K(\sqrt{d})$
for a unique positive divisor $d$ of $N$, and $\chi_d$ will denote the corresponding character of $K$. Let $\varpi=\sqrt{-7N}$, so that $\varpi$ is a
uniformizer of $K_7 = \BQ_7(\sqrt{-7})$ with $\chi^{(N)}(\varpi)=1$.  We now take the primitive Gross-Prasad test vector $f$ for $(A,\chi^{(N)})$ as in Theorem \ref{test vector}.
Again from Lemma \ref{relation for yd}, we see that $y_d=y_{N/d}$ for every $d$ dividing $N$. Let $\CA$ denote the ideal class group of $K,$ and note that $\# (2\CA)$ is always even by Gauss' genus theory.
In fact, the class of $\varpi=\sqrt{-7N}$ always
belongs to $\CA[2]\cap 2\CA$, and $f$ is invariant under multiplying
by $\varpi$. Thus we have that
\begin{equation}\label{x5}
\sum_d y_d=2^{k+1} \sum_{t\in 2\CA/\varpi} f(t).
\end{equation}
where the sum on the left is over all positive divisors $d$ of $N$. In particular, it follows that $\sum_d y_d$ is divisible by $2^{k+1}$.

\medskip

Now we argue by induction on $k$. When $k=1$ and $N = p_1$, we
have $y_1+y_N=2y_1$ is divisible by $4$, i.e. $y_1$ is divisible by
$2$. But
$$
y_1^2=4\LL{1}\LL{-7N}=4\LL{1}\LL{N}=2\LL{N}.
$$
It follows that $ord_2(\LL{N})\geq 1$, as required. Next assume that
$N=p_1\cdots p_k$ with $k\geq 2$. For each positive divisor $d$ of $N$, with
$d\neq 1, N$, we have by induction that
$$
y_d^2=4\LL{d}\LL{-7N/d}=4\LL{d}\LL{N/d}
$$
has $2$-adic valuation at least $2k$, i.e. $y_d+y_{n/d}=2y_d$ has
$2$-adic valuation at least $k+1$. It follows from \eqref{x5} that
$2y_1=y_1+y_N$ has $2$-adic valuation at least $k+1$, whence
$ord_2(y_1)\geq k$. Thus
$$
y_1^2=4\LL{1}\LL{-7n}=2\LL{N}
$$
has valuation at least $2k$, and $ord_2(\LL{N})\geq 2k-1$, completing the proof.
\end{proof}

Let $D$ be the discriminant of an imaginary quadratic field, and let  $\mathcal{A}(D)$ be the ideal class group of this field. For each integer $i \geq 1$, we define the $2^i$-th rank of
its ideal class group by
$$
h_{2^i}(D) = \BF _2- \, {\rm dimension of } \, \, 2^{i-1}\mathcal{A}(D)/2^{i}\mathcal{A}(D).
$$

\begin{prop}\label{s0}Let $p$ be a prime which is $\equiv 1$, modulo $4$, and a quadratic residue modulo $7$.
Then the following assertions are equivalent:
\begin{enumerate}
 \item $h_8(-7p)=0$;
 \item $p$ does not split completely in $\BQ(A[4])=\BQ(i, \sqrt[4]{-7})$;
 \item $ord_2(\LL{p})=1$;
 \item the quotient of the $2$-Selmer group of $A^{(p)}$ by the image of $A^{(p)}(\BQ)[2]$
 is trivial.
\end{enumerate}
\end{prop}

\begin{proof}
From Propositions \ref{descent1} and \ref{descent2}, we conclude that
$$
S^{(\phi)}(A^{(p)}/\mathbb{Q})= \{1,-7\}, \, \textrm{or} \,  \, \{1,-7,p,-7p\},
$$
and
$$
S^{\hat{(\phi)}}(A^{'(p)}/\mathbb{Q})=\{1,7\}, \, \textrm{or} \, \, \{1,7,p,7p\},
$$
according as $\left(\frac{-7}{p}\right)_{4}\neq1, \, \textrm{or} \, \left(\frac{-7}{p}\right)_{4}=1$.
Recall that $\mathfrak S^{(2)}(A^{(p)}) = S^{(2)}(A^{(p)}/\mathbb{Q}) / A^{(p)}(\mathbb{Q})[2]$.
Note also the parity theorem of the Dokchitser brothers implies that
$\mathfrak S^{(2)}(A^{(p)})$ has even dimension. Hence we deduce from
the exact sequence \eqref{d11} that, if $p$ is a prime with
$p\equiv1 \mod 4$ and $\left(\frac{p}{7}\right)=1$, then
$$
\textrm{dim}_{\mathbb{F}_{2}}\mathfrak S^{(2)}(A^{(p)})=
\begin{cases}0, &\textrm{if} \  \left(\frac{-7}{p}\right)_{4}\neq1\\
2, &\textrm{if}\ \left(\frac{-7}{p}\right)_{4}=1
\end{cases}.
$$
Secondly, we recall some general facts about the computation of
the $8$-rank of a number field $K_{0}=\BQ(\sqrt{D})$ with discriminant
$D=p^{*}_{1}\cdots p^{*}_{t}<0$
and ideal class group $\mathcal{A} = \mathcal{A}(D)$, where $p^{*}_{i}:=(-1)^{p_{i}-1}\cdot p_{i}$ if $p_i$ is odd and setting $2^*=-4,8,-8$.
For simplicity, if $2$ divides $D$, we always assume that $p_{t}=2$. Note here we choose the value of $p^*_t$ to guarantee that the discriminant $D$ is always congruent to $0$ or $1$ modulo $4$. By Gauss' genus theory, we know that $h_{2}(D)=t-1$. To study the $4$-rank of $\CA$,
we denote by $R = (R_{ij})$ the R\'{e}dei matrix, which is defined by
$R_{ij} = \left[\frac{p_i^*}{p_j} \right]$ if $i \neq j$ and
$R_{ii} = \sum_{j \neq i} R_{ij}$, where
$(-1)^{\left[ \frac{p_j^*}{p_i} \right]} = \left( \frac{p_j^*}{p_i}\right)$
(if $p_t = 2$, then $\left(\frac{\cdot}{2} \right)$ is the Kronecker symbol).
Now consider the linear equation $RX_{d}=0$ over $\mathbb{F}_{2}$,
where $X_{d}=(x_{1},\ldots,x_{t})^{T}$, and we define $x_{i}=1$ if
and only if $p_{i}$ divides $d$. Then $X_d$ is a solution if and only if
$\left(\frac{d}{p_{i}}\right)=1$ for all odd $p_{i}$ not dividing $d$, and
$\left(\frac{-D/d}{p_{j}}\right)=1$ for all odd $\ p_{j}$ dividing $d$.
This is also equivalent to the assertion that the equation
$dz^{2}=x^{2}-Dy^{2}$ has a non-trivial positive integer solution.
Let $\mathfrak D$ denote the set of square free prime divisors of $D$,
viewed as a subgroup of exponent $2$ in $\BQ^\times/{\BQ^\times}^2$.
Also note the fact that $[a]\in 2\CA$ if and only if
$\text{Norm}([a])\in N_{K_0^{\times}/\BQ^{\times}}(K_0^{\times})$.
Since the map
$$
\mathfrak D\cap N_{K_0^{\times}/\BQ^{\times}}(K_0^{\times}) \rightarrow \{X_{d},RX_{d}=0\}; \ \ \ d \mapsto X_{d}
$$
is an isomorphism, and the surjective map
$$\mathfrak D \cap N_{K_0^{\times}/\BQ^{\times}}(K_0^{\times})\rightarrow \CA[2]\cap 2\CA $$
has kernel of order two, it follows that $h_4(D) = t - 1 -\rank_{\BF_2} R$.

Now we consider under what conditions the $8$-rank of $K_0$ is equal to $1$, assuming that the $4$-rank of $K_0$ is $1$. Because of this last condition, there is a non-zero element $[Q]\in 2\CA\cap\CA[2]$, and the following assertions are easily seen to be equivalent: \\
i) the $8$-rank of $K_0$ is equal to one; \\
ii) $[Q]=[P_c]^2$, with $[P_c]\in2\CA+\CA[2]$, where $[P_c]$ is an ideal of $K_0$ above $(c)\subset \BQ$; \\
iii) the linear equation $RX=C$ has solution in $\BF_2$, with $R$ the R\'{e}dei matrix and $C=([\frac{c}{p_1}], \ldots, [\frac{c}{p_t}])^{T}$.\\

We now use the above to compute the $8$-rank of the number field $K=\BQ(\sqrt{-7p})$ when $p\equiv1 \mod 4$ and $p$ is a square modulo $7$. We claim that
$h_{8}(-7p)=1$ if and only if $\left(\frac{-7}{p}\right)_{4}=1$.
Indeed, since $D=-7p$, the R\'{e}dei matrix $R$ is of the following form
$$
  \begin{pmatrix}
    [\frac{p}{7}] &[\frac{-7}{p}] \\
    [\frac{p}{7}] & [\frac{-7}{p}] \\
  \end{pmatrix}=0.
$$
Assuming $h_{4}(-7p)=1$, let the basis element be
$[P]\in 2\mathcal {A}\cap \mathcal{A}[2]$, where $[P]^{2}=(p)$.
From the above we know that $z^{2}=px^{2}+7y^{2}$ will have a
relatively prime integer solution $(a,b,c)$ over $\BZ$.
Note also the fact that $[Q]\in 2\mathcal {A}$ if and only if
$q=\text{Norm}([Q]) \in N_{K/\BQ}(K^{\times})$. Hence
$[P]=[P_{c}]^{2}$, and so we will have $h_{8}(-7p)=1$
if and only if the equation $z^{2}=px^{2}+7y^{2}$ has
a non-zero positive integer solution $x=a, \ y=b, \ z=c$,
with $\left(\frac{c}{7}\right)=\left(\frac{c}{p}\right)=1$.
We may clearly assume that $(a,b,c)$ are relatively prime.
Considering this equation modulo $4$, we can see that $a$ must
be odd. Noting the identity $p(3a)^{2}=(4c+7b)^{2}-7(c+4b)^{2}$,
we can also assume that $c$ is odd, whence
$\left(\frac{c}{7}\right)=\left(\frac{c}{p}\right)$.
Hence we can consider the equation in the following two cases.
If $p \equiv 1 \mod 8$, then the Jacobi symbol $\left(\frac{b}{p}\right)=1$,
so from
$$
\left(\frac{c}{p}\right)=\left(\frac{c^{2}}{p}\right)_{4}=\left(\frac{7}{p}\right)_{4}\cdot\left(\frac{b}{p}\right)=\left(\frac{7}{p}\right)_{4},
$$
we conclude that $h_{8}(-7p)=1$ if and only if
$\left(\frac{7}{p}\right)_{4}=1$.
If $p\equiv5 \mod 8$, then by considering the equation
in $\BQ_{2}$, we see that $b \equiv 2 \mod 4$, whence the
equation can be rewritten as
$$
c^{2}=pa^{2}+4\cdot7\cdot b'^{2},
$$
showing that $\left(\frac{b'}{p}\right)=1$, where $b'=b/2$. Thus, as
$$
\left(\frac{c}{p}\right)=\left(\frac{28b^{'2}}{p}\right)_{4}=\left(\frac{28}{p}\right)_{4}\cdot\left(\frac{b'}{p}\right)=\left(\frac{-7}{p}\right)_{4},
$$
it follows that $h_{8}(-7p)=1$ if and only if
$\left(\frac{-7}{p}\right)_{4}=1$. In conclusion,
we have shown that $h_{8}(-7p)=0$ if and only if
$\left(\frac{-7}{p}\right)_{4}\neq1$,
and this completes the proof of the equivalence of (1) and (4).
Note also the equivalence of (1) and (3) can be seen from
the proof of the above proposition, using the fact that
$$
2y_{p}=2^{2}\sum_{t\in 2\CA/<\varpi>}f(t).
$$
Now $f$ is supported on $2\CA$, takes values in $\pm1$, and note also that $h_{8}(-7p)=0$, so $2\CA/<\varpi>$ has odd cardinality. Thus $ord_2(y_{p})=1$, and the assertion follows.
Moreover, the equivalence of (2) and (4) is clear, and the proof of
the proposition is complete.
\end{proof}

\begin{prop}\label{s}Let $N=p_1\cdots p_k$ be a product of
$k \geq 1$ distinct primes which split completely in
$\BQ(i, \sqrt[4]{-7})$. Then $ord_2( \LL{N}) \geq 2k+1$.
\end{prop}

\begin{proof}
For any positive integer $d$ dividing $N$, the ideal class group
$\CA=\CA(-7d)$ of the imaginary quadratic field $\BQ(\sqrt{-7d})$ always has
$4$ dividing $\# (2\CA)$. In fact, when $d$ is a prime, $\BQ(\sqrt{-7d})$ has
$8$-rank $1$ by Proposition \ref{s0}. When $d$ is not a prime, then
either $\BQ(\sqrt{-7d})$ has $4$-rank at least $2$ or $8$-rank
at least $1$. We only need to explain the latter case, say that
we have $r_{8}(-7d)=1$ with $r_{4}(-7d)=1$. From the linear equation $RX=0$, we conclude that
$[\sqrt{-7}]\in \CA[2]\cap2\CA$. We claim that $[\sqrt{-7}]$ is in $\CA[8]$.
According to the equivalent assertions in the proof of Proposition \ref{s0}, this will be true
if and only if the equation $z^{2}=7x^{2}+dy^{2}$, has a non-trivial
positive integer solution $(a,b,c)$, with the integer $c$ satisfying
$\left(\frac{c}{7}\right)=1$. This can be shown as follows.
Without loss of generality, we can assume $b,c$ are odd, and
$\left(\frac{c}{7}\right)=\left(\frac{c}{d}\right)$. We then conclude that,
when $d \equiv5 \mod 8$, we have
$$
db^{2}=c^{2}-7\cdot4{a'}^{2},
$$
where $a'=a/2$ is odd. Thus
$$
\left(\frac{c}{7}\right)=\left(\frac{c}{d}\right)=\left(\frac{7\cdot4}{d}\right)_{4}\cdot\left(\frac{a'}{d}\right)=\left(\frac{-7}{d}\right)_{4}=1.
$$
When $d \equiv 1 \mod 8$, then from
$$c^2=7a^2+db^2$$
we know $\left(\frac{a}{d}\right)=1$, so
$$\left(\frac{c}{d}\right)=\left(\frac{c^2}{d}\right)_4=\left(\frac{7}{d}\right)_4\left(\frac{a}{d}\right)=\left(\frac{-7}{d}\right)_4=1,$$
which proves our claim.

The proposition then follows by a similar
induction argument as in the proof of Proposition \ref{w}.
\end{proof}

In the final part of this section, we give an another proof of the result which is needed in the inductive argument with Heegner points in the last section of the paper. Note that we are able to prove a slightly stronger result than that given by Zhao's method in Theorem \ref{main4}. Recall that the field $\mathfrak H$ is defined by \eqref{t2}.

\begin{thm}\label{bw}
Let $M=q_{1}q_{2}\cdots q_{r}p_{1}p_{2}\cdots p_{k}$ be a
product of distinct primes, where $r \geq 0$ and $k \geq 1$. Assume that
\begin{enumerate}
  \item for $1 \leq i \leq r$, we have $q_{i}\equiv 1 \mod 4$, and $q_i$ is inert in the field $\BQ(\sqrt{-7})$,
  \item for $1 \leq j \leq k$, we have $p_{j}\equiv 1 \mod 4$, and $p_j$ splits completely in
  the field $\mathfrak H$.
\end{enumerate}
Then $ord_2(L^{(alg)}(A^{(M)},1)) \geq 2k+r+1$.
\end{thm}

\begin{proof}
Put $R_{+}=q_{1}q_{2}\cdots q_{r}$, $N_{+}=p_{1}p_{2}\cdots p_{k}$,
and take $K=\BQ(\sqrt{-7M})$. We shall assume $r \geq 1$, since the
case when $r=0$ has been considered in the previous proposition. We first note that,
under the conditions of the theorem, $4$ divides $\#(2\CA)$, where
$\CA$ now denotes the ideal class group of $K=\BQ(\sqrt{-7M})$. This is because when we look at the R\'{e}dei matrix, the conditions that we put on $p_j$'s force this matrix be a block matrix with respect to the $q_i$ and $p_j$'s.

\medskip

In the following argument, we primarily use induction on $r$, and for each fixed $r$, we also use induction on $k$.
We first assume $r=1$. Then $K=\BQ(\sqrt{-7q_1N_{+}})$, and we see that
$K_{7}\cong\BQ_{7}(\sqrt{-35}), \chi^{(q_1N_{+})}(\varpi)=-1$. Take the primitive Gross-Prasad test vector $f$ for
$(A,\chi^{(qN_{+})})$ according to Theorem \ref{test vector}. We now use induction on $k$, starting with $k=1$. We then have
$$
y_{p_1q_1}+y_{q_1}=\sum_{d\mid p_1q_1}y_d=2^2\sum_{t\in 2\CA}f(t).
$$
In the above formula, the terms $y_{p_1}$ and $y_1$ vanish, because of our choice of the test vector. By Theorem \ref{ord2LR+} and Proposition \ref{s}, we have $y^2_{q_1}=8L^{(alg)}(A^{(q_1)},1)L^{(alg)}(A^{(p_1)},1)$, whence $ord_2(y_{q_1})\geq 3$. Note that $f(t)$ take value $\pm1$ on $2\CA$ and $\#(2\CA)$ is even, and thus the right hand side of the above formula has $2$-adic valuation at least $3$. It follows that $ord_2(y_{p_1q_1})\geq3$. Then by $y^2_{p_1q_1}=8L^{(alg)}(A^{(p_1q_1)},1)L^{(alg)}(A,1)$, we deduce that $ord_2(L^{(alg)}(A^{(p_1q_1)},1))\geq 4$, which proves our assertion for $k=1$.
For
$k \geq 2$, we have
$$
\sum_{d\mid N_{+}}y_{q_1d}=\sum_{d\mid qN_{+}}y_{d}=2^{k+1}\sum_{t\in 2\CA}f(t).
$$
We can rewrite the left term as
$$
y_{q_1N_{+}}+y_{q}+\sum_{d\mid N_{+} \ d\neq 1,N_{+}}y_{dq_1}=2^{k+1}\sum_{t\in 2\CA}f(t).
$$
By the induction hypothesis, Theorem \ref{ord2LR+} and
Proposition \ref{s}, we know that for $d\neq1,N_{+}$, we have
$$
y^{2}_{dq_1}=8L^{(alg)}(A^{(dq_1)},1)L^{(alg)}(A^{(N_{+}/d)},1),
$$
whence $ord_2(y_{d})\geq k+3$. We also have
$$
y^{2}_{q_1}=8L^{(alg)}(A^{(q_1)},1)L^{(alg)}(A^{(N_{+})},1),
$$
which gives $ord_2(y_{q_1})\geq k+2$. Noting that $ord_2(2^{k+1}\sum_{t\in 2\CA}f(t)) \geq k+2$,
we conclude that $ord_2(y_{q_1N_{+}})\geq k+2$. But by Waldspurger's formula, we have
$$
y^{2}_{dN_{+}}=8L^{(alg)}(A^{(q_1N_{+})},1)L^{(alg)}(A,1),
$$
and so we obtain $ord_2(L^{(alg)}(E^{(q_1N_{+})},1))\geq 2k+2$.
This proves the case for $r=1$.

\medskip

Now we assume $r > 1$, and we break up the argument into two cases,
according that $r$ is even or odd. Suppose first that $r$ is odd,
so that $r \geq 3$. In this case, $K_{7}\cong\BQ_{7}(\sqrt{-35}), \chi^{(M)}(\varpi)=-1$, and we can take the
test vector $f$ for $(A,\chi^{(M)})$ according to Theorem \ref{test vector}. We use
induction on $k$. Then we have
$$
\sum_{1\leq d \mid M,\tau((d,R_{+})) \ \textrm{odd}}y_{d}=\sum_{1\leq d\mid M}y_{d}=2^{r+k}\sum_{t\in 2\CA}f(t).
$$
We analyse all the terms except $y_{M}$ in what follows.
Firstly,
$$
y^{2}_{s}=8L^{(alg)}(A^{(s)},1)L^{(alg)}(A^{(\frac{R_{+}}{s}\cdot N_{+})},1), \ \forall s\mid R_{+}, \tau(s) \  \textrm{odd},
$$
gives $ord_2(y_{s})\geq k+\frac{r+3}{2}$,
while
$$
y^{2}_{st}=8L^{(alg)}(A^{(st)},1)L^{(alg)}(A^{(\frac{R_{+}}{s}\cdot \frac{N_{+}}{t})},1), \ \forall s\mid R_{+}, \tau(s) \  \textrm{odd}, 1\neq t\mid N_{+},st\neq M,
$$
gives $ord_2(y_{st})\geq k+\frac{r+5}{2}$.
Note also that $ord_2(2^{r+k}\sum_{t\in 2\CA}f(t))\geq r+k+1$.
Putting these results together, we obtain
$ord_2(y_{M})\geq k+\frac{r+3}{2}$, whence Waldspurger's
formula gives $ord_2(L^{(alg)}(A^{(M)},1))\geq 2k+r+1$, as required.

Assume next that $r$ is even, so that $r\geq2$. Then $K_{7}\cong\BQ_{7}(\sqrt{-7}), \chi^{(M)}(\varpi)=1$, and we again
take the primitive test vector $f$ for $(A,\chi^{(M)})$ according to Theorem \ref{test vector}. We again apply induction on $k$. Now we have
$$
\sum_{1\leq d\mid M,\  \tau((d,R_{+})) \ \textrm{even}}y_{d}=\sum_{1\leq d\mid M}y_{d}=2^{r+k}\sum_{t\in 2\CA}f(t).
$$
Note that, in this case, from Lemma \ref{relation for yd} we have $y_{d}=y_{n/d}$, and so the above summation can be rewritten as
$$\sum_{1\leq d\mid M,\  \tau((d,R_{+})) \ \textrm{even}, \  d\geq\sqrt{R_M}}y_{d}=2^{r+k-1}\sum_{t\in2\CA}f(t).$$
We again analyse all the terms in above formula except $y_{M}$.
Firstly,
$$
y^{2}_{s}=4L^{(alg)}(A^{(s)},1)L^{(alg)}(A^{(\frac{R_{+}}{s}\cdot N_{+})},1), \ \forall s\mid R_{+}, \tau(s) \  \textrm{even},
$$
gives $ord_2(y_{s})\geq k+\frac{r}{2}+1$, while
$$
y^{2}_{st}=4L^{(alg)}(A^{(st)},1)L^{(alg)}(A^{(\frac{R_{+}}{s}\cdot \frac{N_{+}}{t})},1), \ \forall s\mid R_{+}, \tau(s) \  \textrm{even},
1\neq t\mid N_{+},st\neq M,
$$
gives $ord_2(y_{st})\geq k+\frac{r}{2}+2$.
Note also $ord_2(2^{r+k-1}\sum_{t\in 2\CA}f(t))\geq r+k$.
Putting all together, we obtain
$ord_2(y_{M})\geq k+\frac{r}{2}+1$,
whence, using Waldspurger's formula, we conclude that
$ord_2(L^{(alg)}(A^{(M)},1))\geq 2k+r+1$.
This completes the proof.
\end{proof}

\section{Heegner Points for Quadratic Twists}
In this section, we show that one can combine the induction method as in \cite{Tian2},
\cite{Tian}, and the method in section 2, to give a fairly general result about the existence of many
quadratic twists of the elliptic curve $A=(X_0(49), [\infty])$ with Heegner points of infinite order. The following theorem is just a restatement of Theorem \ref{main2}.
We recall that $F=\BQ(\sqrt{-7})$.

\begin{thm}\label{mainf} Let $M=-\ell_0 RN$ be a negative square-free
integer, prime to $7$, such that
\begin{enumerate}
\item $\ell_0 > 3$ is a prime which is $\equiv 3\mod 4$ and is inert in $F$,
\item $R$ is a product of primes which are $\equiv 1 \mod 4$, and which are inert both in
$F$ and in $\BQ(\sqrt{-\ell_0})$,
\item $N$ is a product of primes which split completely in $\BQ(A[4])$, and in the fields
$\BQ(\sqrt{q})$, for $q$ running over all primes dividing $R$,
\item the ideal class group of $K_N=\BQ(\sqrt{-\ell_0 N})$ has no element of order
$4$.
\end{enumerate}
Then $L(A^{(M)}, s)$ has a simple zero at $s=1$, $A^{(M)}(\BQ)$ has rank $1$, and the Tate
Shafarevich group of $A^{(M)}(\BQ)$ is finite of odd cardinality.
\end{thm}

We prove Theorem \ref{mainf} by constructing non-trivial Heegner
points as follows. Assume that $M=-\ell_0RN$ is an integer
satisfying the conditions (1), (2), and (3) of Theorem \ref{mainf}. Note that
$\BQ(A[4])=\BQ(i, \sqrt[4]{-7})$, and that $7$, the unique prime dividing the
conductor of $A$, is split in $K_N=\BQ(\sqrt{-\ell_0 N})$. Thus the
classical Heegner hypothesis holds for $A$ and the imaginary quadratic field $K_N$. Let $H_{R,N}$ be the
ring class field of $K_N$ of conductor $R$, and define $\mathfrak J_{R,N}$ by
$$
\mathfrak J_{R,N} = K_N(\sqrt{-\ell_0}, \sqrt{q_1}, \ldots, \sqrt{q_r}, \sqrt{p_1}, \ldots, \sqrt{p_k}),
$$
where the $q_i$ (resp. the $p_j$) run over the distinct prime factors of $R$ (resp. $N$). By the theory of genera, and similar arguments
to those given in the proof of Lemma \ref{sub}, we see that $\mathfrak J_{R,N}$ is a subfield of $H_{R,N}$. In particular, it follows that
$K_N(\sqrt{RN})\subset H_{R,N}$. Let $P_{R,N} \in A(H_{R,N})$ be the Heegner point of
conductor $R$ attached to $A=X_0(49)$ and the imaginary quadratic field $K_N$ in exactly the same fashion as in section 2, and let $\chi_{R}$ be the quadratic character defining the extension $K_N(\sqrt{R})/K_N$. We define the Heegner point
$$
Y_{R,N}=\sum_{\sigma\in \Gal(H_{R,N}/K)} \chi_R(\sigma)\sigma(P_{R,N})\in A(K_N(\sqrt{R})).
$$
Theorem \ref{mainf} then follows from Theorem \ref{M5} and the following result.

\medskip

\begin{thm}\label{mainf1} Let $M=-\ell_0 RN$ be a negative square-free
integer, prime to $7$, satisfying conditions (1), (2), and (3) of Theorem \ref{mainf}.
Then the Heegner point $Y_{R,N} \in A(K_N(\sqrt{R}))$ satisfies
$$Y_{R,N}\in 2^{k+r-1 }A(\BQ(\sqrt{M}))^-+A(\BQ(\sqrt{M}))_\tor,$$
where $r$ and $k$ denote the number of prime factors of $R$ and $N$, respectively, and, when $r=k=0$, this should be interpreted as meaning that $2Y_{R,N}\in A(\BQ(\sqrt{M}))^-$. If, in addition, the ideal class group of $K_N$ has no element of order $4$,
then $$Y_{R,N}\notin 2^{k+r}A(\BQ(\sqrt{M}))^-+A(\BQ(\sqrt{M}))_\tor,$$
whence $Y_R$ is of infinite order.
\end{thm}

\begin{proof} We argue by induction on $k$, the number of prime factors of $N$. The initial case with $k=0$
is given by Birch's lemma when $r=0$, and by Theorem \ref{Birch2} when $r \geq 1$ (note that any prime dividing
$R$ will be a sensitive supersingular prime for $A$, which is inert in $\BQ(\sqrt{-\ell_0})$). Now
assume that $k\geq 1$. For each positive divisor $D$ of $NR$,
let $\chi_D$ be the character over $K_N$ defining the extension
$K_N(\sqrt{D})$ and consider the associated imprimitive Heegner point in $A(K_N(\sqrt{D})$ given by
$$Z_{D,N}=\sum_{\sigma\in \Gal(H_{R,N}/K)} \chi_D(\sigma)\sigma(P_{R,N}).$$
Thus $Y_{R,N}= Z_{R,N}$. We also define
$$
\Psi_{R,N} = Tr_{H_{R,N}/\mathfrak J_{R,N}}(P_{R,N}).
$$
\begin{lem} Let $D$ be any positive divisor of $RN$. Then $Z_{D,N}=0$ unless $R$ divides $D$.
Moreover, we have
\begin{equation}\label{f1}
\sum_{R|D|RN} Z_{D,N}=2^{k+r}\Psi_{R,N},
\end{equation}
where the sum is taken over all positive divisors $D$ of $RN$ which are divisible  by $R$.
\end{lem}
\begin{proof}
Let $D$ be any positive divisor of $RN$, which is not divisible by $R$. Thus there exists a prime
$q$ dividing $R$, which does not divide $D$. Put $R' = R/q$. Then, by the Kolyvagin property for Heegner points, we have
$$
Tr_{H_{R,N}/H_{R',N}}(P_{R,N}) = a_qP_{R',N} = 0,
$$
because $a_q = 0$. Hence, since $D$ divides $R'$, we have
$$
Z_{D,N}=\sum_{\sigma \in \Gal(H_{R',N}/K)} \chi_D(\sigma)\sigma(\RTr_{H_{R,N}/H_{R',N}}(P_{R,N})) = 0.
$$
Thus the sum on the left hand side of \eqref{f1} is given by
$$
\sum_{D|RN}Z_{D,N}=\sum_{\sigma\in \Gal(\mathfrak J_{R,N}/K)}\left(\sum_{
D|RN}\chi_D(\sigma)\right)\sigma(\Psi_{R,N})=2^{k+r} \Psi_{R,N}.
$$
\end{proof}
Since $Y_{R,N} = Z_{R,N}$, we conclude from this lemma that
\begin{equation}\label{f6}
Y_{R,N}=-\sum_{1 < d|N} Z_{dR, N} +2^{k+r} \Psi_{R,N},
\end{equation}
where $d$ runs over the divisors of $N$ which are $>1$.
For each such $d$,  it is not difficult to see that  $Z_{dR, N}\in
A(\BQ(\sqrt{-\ell_0RN_d}))^-$, where $N_d = N/d$. Our strategy is to use an induction argument to prove that
\begin{equation}\label{f2}
Z_{dR,N}\in 2^{k+r} A(\BQ(\sqrt{-\ell_0RN_d}))^-+ A(\BQ(\sqrt{-\ell_0RN_d}))_\tor
\end{equation}
for all divisors $d > 1$ of $N$.
Let $K_{N_d}=\BQ(\sqrt{-\ell_0
N_d})$ and construct the analogous point $Y_{R, N_d}$,
with $N$ replaced by its divisor $N_d$, but keeping the same value of $R$. Now both $Z_{dR,N}$ and $Y_{R,N_d}$ belong to $W_d = A(\BQ(\sqrt{-\ell_0RN_d}))^-$, and we can compare their heights by using Theorem \ref{M5} and Kolyvagin's theorem. Note that
\begin{equation}\label{f3}
L(A/K_{N}, \chi_{dR}, s)/L(A/K_{N_d},\chi_R,s) = L(A^{(dR)},s)/L(A^{(R)},s).
\end{equation}
Since $L(A^{(R)},s)$ does not vanish at $s=1$ by Theorem \ref{main3}, it follows from Theorem \ref{M5} that $Z_{dR,N}$ has finite order if $Y_{R, N_d}$ has finite order, and then
\eqref{f2} is plainly true. On the other hand, if $Y_{R, N_d}$ is of infinite order, and $Z_{dR,N}$ is not torsion, Theorem \ref{M5} implies that $W_d$ has rank 1, and that
\begin{equation}\label{f4}
\frac{\wh{h}_{K_N}(Z_{dR, N})}{\wh{h}_{K_{N_d}}(Y_{R, N_d})}=\frac{L^{(alg)}(
A^{(dR)}, 1)}{L^{(alg)}(A^{(R)},1)},
\end{equation}
with $L(A^{(dR)}, 1) \neq 0$. By the strong form of Theorem \ref{main3} given by Theorem \ref{ii}, and Theorem \ref{main4}, we conclude that
\begin{equation}\label{f5}
ord_2(\frac{L^{(alg)}(A^{(dR)}, 1)}{L^{(alg)}(A^{(R)},1)}) \geq 2k(d) +1,
\end{equation}
where $k(d)$ denotes the number of prime factors of $d$. Now, since $d > 1$, we may assume by induction on $k$ that $Y_{R, N_d}$ belongs to $2^{k-k(d)-1}W_d + A(\BQ(\sqrt{-\ell_0RN_d}))_\tor$, and so we conclude from \eqref{f4} and \eqref{f5} that \eqref{f2} is indeed valid for all divisors $d > 1$ of $N$.
Hence we deduce from \eqref{f6} that
$$
Y_{R,N} \in 2^{k+r} A(\mathfrak J_{R,N}) + A(\mathfrak J_{R,N})_\tor.
$$
Similar arguments with Galois cohomology to those given in the proof of Theorem \ref{Birch2} then show that
$$
Y_{R,N}\in 2^{k+r-1} A(\BQ(\sqrt{M}))^-+A(\BQ(\sqrt{M}))_\tor.
$$
Finally, if $K_N=\BQ(\sqrt{-\ell_0 N})$ has no elements of order
$4$ in its ideal class group, then the degree $[H_{R,N} : \mathfrak J_{R,N}]$ is odd and we have that
$\Psi_{R,N}+\ov{\Psi}_{R,N}=T$ the non-trivial  point of order $2$ in $A(\BQ)$. An entirely
similar argument to that given in the proof of Theorem \ref{Birch2} then shows that
$$Y_{R,N} \notin 2^{k+r}A(\BQ(\sqrt{M}))^-+A(\BQ(\sqrt{M}))_\tor.$$
This completes the proof of the theorem.

\end{proof}

\bigskip

\noindent John Coates,\\
Emmanuel College, Cambridge, and\\
Department of Mathematics, POSTECH, Korea.\\
{\it jhc13@dpmms.cam.ac.uk }

\bigskip

\noindent Yongxiong Li, \\
Academy of Mathematics and Systems Science, \\
Morningside center of Mathematics, \\
Chinese Academy of Sciences, Beijing 100190, China. \\
{\it liyx\_1029@126.com }

\bigskip

\noindent Ye Tian, \\
Academy of Mathematics and Systems Science, \\
Morningside center of Mathematics, \\
Chinese Academy of Sciences, Beijing 100190, China. \\
{\it ytian@math.ac.cn }

\bigskip

\noindent Shuai Zhai, \\
School of Mathematics, Shandong University, \\
Jinan, Shandong 250100, China, and \\
Department of Pure Mathematics and Mathematical Statistics, \\
University of Cambridge, UK. \\
{\it shuaizhai@gmail.com}


\begin{thebibliography}{XXXX}

\addtocontents{Bibliography}

\bibitem{BFH} D. Bump, S. Friedberg and J. Hoffstein, {\em Non-vanishing theorems for $L$-functions of modular forms and their derivatives}, Invent. Math. 102 (1990), 543-618.

\bibitem{CST} L. Cai, J. Shu and Y. Tian, {\em Explicit Gross-Zagier Formula and Waldspurger Formula}, preprint, 2013.

\bibitem{C2} J. Coates, {\em Lectures on the Birch-Swinnerton-Dyer conjecture}, Notices of the ICCM, 2013.

\bibitem{CKLZ} J. Coates, M. Kim, Z. Liang and C. Zhao, {\em On the $2$-part of
the Birch-Swinnerton-Dyer conjecture for elliptic curves with
complex multiplication}, Munster J. of Math. 7 (2014), 83-103.

\bibitem{C} J. Coates, {\em Elliptic curves with complex multiplication and Iwasawa theory},
Bull. London Math. Soc. 23 (1991), 321-350.

\bibitem{D} T. Dokchitser and V. Dokchitser, {\em On the Birch-Swinnerton-Dyer quotients modulo squares}, Ann. of Math. 172 (2010), 567-596.

\bibitem{GS} C. Goldstein and N. Schappacher, {\em Series d'Eisenstein et fonctions $L$ de courbes elliptiques a multiplication complexe}, Crelle 327 (1981), 184-218.

\bibitem{Gross} B. Gross, {\em Heegner Points on $X_0(N)$}, in {\em Modular Forms} (ed.
R. A. Rankin). Ellis Horwood (1984).

\bibitem{Pal} V. Pal, {\em Periods of quadratic twists of elliptic curves}, Proceedings AMS 140 (2012), 1513-1525.

\bibitem{RKM} K. Murty and R. Murty, {\em Mean values of derivatives of modular $L$-series}, Ann. of Math. 133 (1991), 447-475.

\bibitem{R} K. Rubin, {\em The main conjectures of Iwasawa theory for imaginary quadratic fields}, Invent. Math 103 (1991), 25-68.

\bibitem{Sil} J. Silverman, {\em The arithmetic of elliptic curves}, Grad. Texts Math. 106, 1986, Springer.

\bibitem{Tian2} Y. Tian, {\em Congruent numbers with many prime factors}, Proc. Natl. Acad. Sci. USA 109 (2012), 21256-21258.
\bibitem{Tian} Y. Tian, {\em Congruent Numbers and Heegner Points}, Cambridge Journal of Mathematics, 2 (2014), 117-161.

\bibitem{TYZ} Y. Tian, X. Yuan and S. Zhang, {\em Genus Periods, Genus Points and Congruent Number Problem}, preprint, 2013.

\bibitem{Tunnell} J. Tunnell, {\em On the local Langlands conjecture for $\GL(2)$}, Invent. Math. 46 (1978), 179-200.

\bibitem{Vigneras} M-F. Vigne\'{r}as, {\em Arithmetique des Algebres de Quaternions}, Lecture Note of Mathematics 800.

\bibitem{Waldspurger} J-L. Waldpsurger, {\em Sur les coefficients de Fourier des formes modulaires de poids demi-entier}
J. Math. Pures Appl. 60 (1981), 375-484.

\bibitem{YZZ} X. Yuan, S. Zhang and W. Zhang, {\em The Gross-Zagier Formula on Shimura
Curves}, Annals of Mathematics Studies Number 184, 2012.

\end{thebibliography}
\end{document}